\documentclass[a4paper]{scrartcl}
\setkomafont{disposition}{\normalfont}%
\usepackage[english]{babel}
\usepackage[T1]{fontenc}
\usepackage[utf8]{inputenc}
\usepackage{graphicx}%
\usepackage{amsmath,amssymb,amsfonts,amscd}%
\usepackage{amsthm}%
\usepackage{mathrsfs}%
\usepackage{xcolor}%
\usepackage{bbm}
\usepackage{hyperref}
\hypersetup{
    colorlinks=true,
    linkcolor=blue,
    urlcolor=blue,
    citecolor=blue
    }
\usepackage{afterpage}
\usepackage{subcaption}
\usepackage{tikz}
\usetikzlibrary{fit,arrows, arrows.meta, fit, positioning, shapes, shapes.geometric, backgrounds}
\usepackage{tikz-3dplot}
\usepackage{pgfplots}
\pgfplotsset{compat=1.18} 
\usetikzlibrary{decorations.pathreplacing}
\usetikzlibrary{svg.path}
\usetikzlibrary{calc}
\usepackage{tikz-cd}
\usepackage{float}
\usepackage{todonotes}
\usepackage{algorithm2e}
\usepackage{enumerate}
\usepackage{adjustbox}
\usepackage{rotating}
\usepackage{extarrows}
\usepackage{booktabs}
\usepackage{csquotes}
\usepackage[backref=false,
    backend=biber,
    style=numeric,
    natbib=true,
    doi=true,
    eprint=true,
    isbn=false,
    url=false]{biblatex}
\usepackage{floatpag}
\floatpagestyle{empty}
\pgfplotsset{compat=1.15}

\usetikzlibrary {patterns,patterns.meta} 
\usetikzlibrary{arrows}

\definecolor{yqqqqq}{rgb}{0.5019607843137255,0,0}
\definecolor{ffqqtt}{rgb}{1,0,0.2}
\definecolor{qqwuqq}{rgb}{0,0.39215686274509803,0}
\definecolor{qqccqq}{rgb}{0,0.8,0}
\definecolor{qqttzz}{rgb}{0,0.2,0.6}
\definecolor{ududff}{rgb}{0.30196078431372547,0.30196078431372547,1}

\newcommand{\eps}{\varepsilon}
\newcommand{\R}{\mathbb{R}}
\newcommand{\N}{\mathbb{N}}
\newcommand{\colim}{\operatorname{colim}}
\newcommand{\Z}{\mathbb{Z}}

\newcommand{\Simp}{\mathbf{Simp}}
\newcommand{\vect}{\mathbf{ Vect}}
\newcommand{\Top}{\mathbf{Top}}
\newcommand{\Rel}{\mathbf{Rel}}

\newcommand{\Nerve}[1]{\operatorname{Nrv}\left({{#1}}\right)}
\newcommand{\Cech}{\mathcal{C}}
\newcommand{\VR}{\mathcal{R}}
\newcommand{\Sd}{\operatorname{Sd}}
\newcommand{\Dow}{\mathcal{D}}
\newcommand{\MDow}{\mathcal{MD}}
\newcommand{\M}{\mathcal{M}}
\newcommand{\DRips}{\mathcal{DR}}
\newcommand{\SRips}{\mathcal{SR}}
\newcommand{\SCech}{\mathcal{SC}}
\newcommand{\ICech}{\mathcal{I}}
\newcommand{\Borel}{\mathfrak{B}}

\renewcommand{\P}{\mathbb{P}}
\newcommand{\hf}[1][M]{\operatorname{hf}^{{#1}}}
\renewcommand{\th}{\textsuperscript{th}}
\newcommand{\Blowup}{\operatorname{Blowup}}
\newcommand{\id}{\operatorname{id}}

\makeatletter
\newcommand{\newreptheorem}[2]{\newtheorem*{rep@#1}{\rep@title}\newenvironment{rep#1}[1]{\def\rep@title{#2 \ref*{##1}}\begin{rep@#1}}{\end{rep@#1}}}
\makeatother

\newtheorem{theorem}{Theorem}[section]
\newreptheorem{theorem}{Theorem}
\newtheorem{corollary}[theorem]{Corollary}
\newtheorem{proposition}[theorem]{Proposition}
\newtheorem{lemma}[theorem]{Lemma}
\theoremstyle{definition}
\newtheorem{definition}[theorem]{Definition}
\newtheorem{example}[theorem]{Example}
\newtheorem{remark}[theorem]{Remark}

\makeatletter
\def\blfootnote{\gdef\@thefnmark{}\@footnotetext}
\makeatother

\title{Density Sensitive Bifiltered Dowker Complexes via Total Weight}
\author{Niklas Hellmer\thanks{Institute of Mathematics, Polish Academy of Sciences, Warsaw. \texttt{niklas.hellmer at impan.pl}}\;\thanks{Faculty of Mathematics, Informatics and Mechanics, University of Warsaw} \and Jan Spali\'nski\thanks{Faculty of Mathematics and Information Science, Warsaw University of Technology}}
\date{\today}

\begin{document}
\maketitle
\begin{abstract}
    In this paper, we introduce new density-sensitive bifiltrations for data using the framework of Dowker complexes.
    Previously, Dowker complexes were studied to address directional or bivariate data whereas density-sensitive bifiltrations on \v{C}ech and Vietoris--Rips complexes were suggested to make them more robust, while increasing computational complexity.
    We combine these two lines of research, noting that the superlevels of the total weight function of a Dowker complex can be identified as an instance of Sheehy's multicover filtration.
    We prove a version of Dowker duality that is compatible with this filtration and show that it corresponds to the multicover nerve theorem.
    As a consequence, we find that the subdivision intrinsic \v{C}ech complex admits a smaller model.
    Moreover, regarding the total weight function as a counting measure, we generalize it to arbitrary measures and prove a density-sensitive stability theorem for the case of probability measures.
    As an application, we propose a robust landmark-based bifiltration which approximates the multicover bifiltration.
    Additionally, we provide an algorithm to calculate the appearances of simplices in our bifiltration and present computational examples.
\end{abstract}



\eject
\section{Introduction}\label{sec:Introduction}
In topological data analysis (TDA), persistent homology of \v{C}ech or Vietoris--Rips complexes is a standard tool to extract information about the shape of data.
There are some shortcomings to this standard approach, notably its lack of sensitivity to density and against bivariate or directional data.
Addressing these issues has been a focus of recent research.
On the one hand, one can introduce a second filtration parameter to capture information about density \cite{blumberg_stability_2022}, just like the proximity parameter of \v{C}ech or Rips controls metric information.
On the other hand, Dowker complexes have received attention in applications involving directional \cite{chowdhury_functorial_2018} or bivariate \cite{yoonDecipheringDiversitySequence2024} data.
In this article, we combine the two approaches, motivated by the following ideas:
\begin{itemize}
    \item A simplex in the \v{C}ech complex corresponds to a certain non-empty intersection of balls, but is not sensitive to how big this intersection is.
    In this context, our construction can be seen as introducing a second axis of filtration, along which simplices persist longer the bigger the intersection is that they correspond to.
    \item The multineighborhood complex of a graph \cite{babsonErdosRenyiGraphsLinialMeshulam2023} is a special case of the total weight filtration.
    It has attracted interest in the setting of random graphs; we provide tools to continue that study in the context of two-parameter persistence.
    \item In two-parameter persistence, one is faced with a trade-off between stability and computability.
    In parallel to our study, models for the subdivision intrinsic \v{C}ech bifiltration were suggested \cite{LesnickNerveModels2024}.
    We present a model that maintains the same asymptotic size with the additional beneift of being a bona-fide bifiltration (as opposed to a semifiltration).
    Furthermore, we suggest the use of landmark-based bifiltrations and provide theoretical foundations for it.
\end{itemize}

In Section~\ref{sec:TotalWeightFiltration}, we start from the total weight function of Robinson \cite{robinson_cosheaf_2022}.
Its superlevel filtration $\Dow(X,Y,R)_\bullet$ turns out to be an instance of  Sheehy's multicover filtration:
\begin{reptheorem}{thm:DowkerDualityTotalWeight}
Let $R\subseteq X\times Y$ be a relation.
Then we have a weak equivalence of filtrations $\vert\Dow(X,Y,R)_\bullet\vert \simeq  \vert\mathcal{S}(\Dow(Y,X,R^\top))_\bullet\vert$, where $\mathcal{S}$ is the subdivision filtration (Definition~\ref{def:SubdivisionBifiltration}).
Moreover, the weak equivalence is natural with respect to filtrations of relations.
\end{reptheorem}
This result extends Dowker's classical duality theorem (as it is recovered at filtration level 1).
We rephrase the total weight function in terms of counting measures, which then paves the way towards a generalisation to arbitrary measures.
In combination with a filtration of relations, this yields the \textit{measure Dowker bifiltration} $\MDow(X,\mu)$ for a set $X$ and a measure $\mu$ in some common ambient space (Definition~\ref{def:MeasureDowkerBifiltration}).
Roughly speaking, the idea is to construct a complex in which data points form a simplex only if there is sufficient mass near it.
The mass can be the point cloud itself, a second point cloud or some ambient measure like Lebesgue's; the meaning of ``near'' can be made precise using a metric but also using any relation, for instance $k$ nearest neighbors (cf. Example \ref{example:GeneExpression}).
We elaborate on the relation  between our construction and other density sensitive bifiltrations.
Notably, the above theorem allows us to identify a model for the subdivision intrinsic \v{C}ech bifiltration with polynomially sized skeleta.

In Section~\ref{sec:Stability}, we prove a stability theorem ascertaining that the change (in the homotopy interleaving distance) of the measure Dowker bifiltration is upper-bounded by the maximum of Hausdorff distance between the data points and Prokhorov metric between the measures:
\begin{reptheorem}{thm:DowkerStability}
    Suppose $(Z,d)$ is a  Polish space, endowed with Borel $\Sigma$-algebra $\Borel(Z)$.
    Let $X_1,X_2\in \Borel(Z)$ and let $\mu_1,\mu_2$ be measures on $(Z,\Borel(Z))$.
    Then we have
    \[
        d_{HI}(\MDow(X_1,\mu_1),\MDow(X_2,\mu_2))\leq \max(\{d_H(X_1,X_2), d_{Pr}(\mu_1,\mu_2)\}),
    \]
    where $d_H$ is the Hausdorff distance (Definition~\ref{def:HausdorffDistance}) and $d_{Pr}$ is the Prokhorov metric (Definition~\ref{def:Prokhorov}).
\end{reptheorem}
As a consequence, we infer a law of large numbers: as the sample size increases, the measure Dowker bifiltration on sample points converges to the true one of the underlying metric probability space (Theorem \ref{thm:LawOfLargeNumbers}).
Moreover, measure Dowker bifiltrations on a fixed set of landmarks are robust (Corollary \ref{cor:RobustnessFixedLandmarks}), they approximate the multicover bifiltration (Corollary \ref{cor:LandmarksApproximateMulticover}), and they are also computationally tractable if the number of landmarks can be chosen to be much smaller than the size of the data point clouds.

Finally, in Section \ref{sec:ComputationalResults}, we present an algorithm (Algorithm~\ref{alg:TWBifilteredDowker}) to compute the measure Dowker bifiltered complex.
We discuss its size and runtime complexity and make an open source implementation available on github\footnote{\url{https://github.com/nihell/pyDowker}}.
We carry out several experiments showcasing applications to protein-ligand binding affinity prediction, clustering and dimensionality reduction of gene expression data and random hypergraphs of Erdös--Renyi type.

Relevant related work includes the study of functorial Dowker duality motivated by TDA in \cite{chowdhury_functorial_2018,robinson_cosheaf_2022,brunRectangleComplexRelation2023}.
The total weight filtration of a Dowker complex was introduced  by Robinson \cite{robinson_cosheaf_2022} and has also been studied in \cite{vaupel_bifiltration_2023}, where it is noted that this filtration is in general different from the one of the dual Dowker complex.
Another approach to bifiltered Dowker complexes \cite{blaserCoreBifiltration2024} was developed in parallel to this work.
Applications of Dowker complexes include protein-ligand binding affinity prediction \cite{liu_dowker_2022, liu_neighborhood_2021}, spatial patterns in the tumor microenvironment \cite{yoonDecipheringDiversitySequence2024}, music theory \cite{freundLatticebasedTopologicalRepresentations2015}, neuroscience \cite{Vaupel2023TopologicalPerspective} and time series and dynamical systems analysis \cite{garlandExploringTopologyDynamical2016}.

For the stability and robustness of two-parameter persistence, \cite{blumberg_stability_2022} is our main reference and inspiration; the work of Scoccola and Rolle \cite{scoccolaLocallyPersistentCategories2020, rolle2023stable} is also of note.
In parallel to this work, Lesnick \& McCabe~\cite{LesnickNerveModels2024} as well as Brun \cite{brun2024dualdegreecechbifiltration} have found related constructions, which will be mentioned in Section \ref{sec:Stability} in particular.

\section{Background}\label{sec:Background}
This article is concerned with establishing a connection between two major recent lines of research in TDA: density sensitive bifiltrations on one hand, and Dowker complexes on the other.
We review both topics briefly and how they arise, starting from the classical idea of \v{C}ech-persistence.

Given a finite point cloud $X\subseteq \R^d$ (which we think of as data arising from measuring $d$ features of a population), its topology might seem uninteresting because it is discrete.
To remedy this, one is led to consider the union of closed balls around the data points $\mathcal{O}(X)_r = \bigcup_{x\in X} \overline{B}_r(x)$.
This is called the \textit{offset filtration}, as we have inclusions $\mathcal{O}(X)_r\hookrightarrow \mathcal{O}(X)_s$ whenever $0\leq r\leq s$.
One can regard this construction as a functor 
\[
\mathcal{O}(X)\colon [0,\infty[\to \Top,
\]
where we view the poset $([0,\infty[,\leq)$ as a thin\footnote{recall that a catergory is called \textit{thin} if between any two objects, there is at most one morphism} category and $\Top$ denotes the category of compactly generated weakly Hausdorff spaces with continuous maps.
This union of closed balls can also be regarded as the sublevel set of the distance-to-$X$ function $d_X(z)=\inf_{x\in X}d(x,z)$ at level $r$.
If the points are perturbed (one can think of a measurement error introduced when repeating the experiment), these sublevel sets do not change much.
This is formalized through the following definition:
\begin{definition}\label{def:HausdorffDistance}
    Let $(Z,d)$ be a metric space and $X,Y\subseteq Z$.
    The \textit{Hausdorff distance} between $X$ and $Y$ is
    \[
    d_H(X,Y) = \sup_{z\in Z}\left|d_Y(z)-d_X(z)\right|.
    \]
\end{definition}
Consequently, for any $\delta>d_H(X,Y)$, we have inclusions
\begin{align*}
    \bigcup_{x\in X} \overline{B}_r(x) 
    &\subseteq \bigcup_{y\in Y} \overline{B}_{r+\delta}(y) 
    \subseteq \bigcup_{x\in X} \overline{B}_{r+2\delta}(x),\\
    \bigcup_{y\in Y} \overline{B}_r(y) 
    &\subseteq \bigcup_{x\in X} \overline{B}_{r+\delta}(x) 
    \subseteq \bigcup_{y\in Y} \overline{B}_{r+2\delta}(y) .
\end{align*}
We say that the unions of balls are $\delta$-\textit{interleaved}.
Their \textit{interleaving distance} $d_I$ is the infimal such $\delta$. (The precise meaning is given below in Definition \ref{def:Interleaving}.)
In order to encode their shape on a computer we need a combinatorial model; namely, abstract simplicial comlplexes.
\begin{definition}
    A set $K$ of finite, non-empty\footnote{Note that Robinson's convention \cite{robinson_cosheaf_2022} allows the empty simplex whereas for us it is more convenient to exclude it.} sets is called an \textit{abstract simplicial complex} if $\tau \in K$ and $\emptyset\neq\sigma \subseteq \tau$ together imply $\sigma \in K$.
\end{definition}
There are several constructions to obtain a simplical complex from data.
In the case considered above, this is given by the nerve of the covering of the offset filtration by the collection of closed balls:
\begin{definition}\label{def:CechComplex}
    The \textit{\v{C}ech complex} $\Cech(X)_r$ is the abstract simplicial complex with vertices $X$ and simplices $\sigma\subseteq X$ whenever $\bigcap_{x\in\sigma} \overline{B}_r(x)\neq \emptyset$.
    We obtain a functor
    \begin{align*}
        [0,\infty[&\to \Simp,\\
        r&\mapsto\Cech(X)_r,\\
        r\leq s &\mapsto \Cech(X)_r \hookrightarrow \Cech(X)_s,
    \end{align*}
    where $\Simp$ denotes the category of abstract simplicial complexes with simplicial maps between them.
\end{definition}
A related construction of a filtered simplicial complex goes as follows:
\begin{definition}\label{def:RipsComplex}
    The \textit{Vietoris--Rips complex} $\VR(X)_r$ is the abstract simplicial complex with vertices $X$ and simplices $\sigma\subseteq X$ whenever we have $d(x,x')\leq r$ for all $x\neq x'\in\sigma$.
    We obtain a functor
    \begin{align*}
        [0,\infty[&\to \Simp,\\
        r&\mapsto\VR(X)_r,\\
        r\leq s &\mapsto \VR(X)_r \hookrightarrow \VR(X)_s.
    \end{align*}
\end{definition}
A third construction is given by Dowker complexes, which we shall introduce more thoroughly below.
By the nerve lemma \cite{Bauer_unified_2023}, the \v{C}ech complex $\Cech(X)_r$ recovers the homotopy type of the offset filtration $\mathcal{O}(X)_r$, see Figure \ref{fig:Cech}.
This holds true not only for every fixed $r$, but in a functorial way respecting the filtration, this is formalized by the notion of weak equivalence of filtrations.

\begin{definition}\label{def:Filtration}
    Let $T$ be a poset regarded as thin category.
    If $\mathbf{C}\in\{\Top,\Simp\}$, a ($T$-indexed) \textit{filtration} is a functor $F\colon T\to\mathbf{C}$ such that all morphisms in $T$ get mapped to inclusions in $\mathbf{C}$.
    We also call $F:T\to \Top$ a \textit{filtered space} and $F\colon T\to \Simp$ a \textit{filtered complex}.
    If $T=]0,\infty[^{op}\times[0,\infty[$, we say $F$ is a \textit{bifiltration.}

    Two filtrations $F,F'\colon T \to \Top$ are said to be \textit{objectwise equivalent} if there is a natural transformation $\eta\colon F\Rightarrow F'$ such that all components $\eta_t\colon F_t\xrightarrow{\simeq}F'_t$ are homotopy equivalences.
    Furthermore, $F,F'$ are called \textit{weakly equivalent}, written $F\simeq F'$, if they are connected via a zig-zag of objectwise equivalences.
    That is, there is a sequence of filtrations
    \[
        F=F^0, F^1, \ldots, F^n=F',
    \]
    such that for each $i\in\{1,\ldots,n\}$ there is an objectwise equivalence $\eta_i\colon F^{i-1}\Rightarrow F^{i}$ or $\eta_i\colon F^{i}\Rightarrow F^{i-1}$.
\end{definition}

Now we can apply homology, which is functorial and homotopy invariant.
As filtrations can be viewed as functors from a poset, postcomposing with homology yields again a functor from the same poset, now valued in vector spaces\footnote{We assume field coefficients throughout}.
Consequently, the homology of the offset filtration can be calculated from the \v{C}ech filtraton using a matrix reduction algorithm.
This leads to the notion of \textit{persistent homology}, which keeps track of homology classes appearing and vanishing as $r$ increases from $0$ to $\infty$.
Algebraically speaking, the persistent homology as a functor decomposes as a direct sum \cite{Crawley-Boevey2015Decomposition}:
\begin{align*}
    H_\ast(\Cech(X))\colon [0,\infty[ &\to k\vect,\\
    H_\ast(\Cech(X))&\cong  \bigoplus\limits_{I\in \mathcal{I}_{H_\ast(\Cech(X))}} kI.
\end{align*}
Here, $\mathcal{I}_{H_\ast(\Cech(X))}$ is a unique multiset of intervals in $[0,\infty[$ associated to ${H_\ast(\Cech(X))}$ and for an interval $I$, the so-called \emph{interval module} $kI$ is defined as 
\begin{align*}
        (kI)_t         & = \begin{cases}k &\text{ if } t\in I,\\ 0 &\text{otherwise},\end{cases}        \\
        (kI)_{s\leq t} & = \begin{cases}\id_k &\text{ if } s,t\in I,\\ 0 &\text{otherwise}.\end{cases}.
    \end{align*}
We think of each interval as corresponding to a topological features which appears (or: is born) in the filtration at the left endpoint of the interval and vanishes (or: dies) at the right endpoint.
This decomposition result unfortunately does not generalise to the multiparameter setting.

Moreover, any $\delta$-interleaving of unions of balls gives rise to a $\delta$-interleaving of persistent homology of \v{C}ech complexes.
In particular, the interleaving distance in persistent homology is upper-bounded by the Hausdorff distance of the point clouds, this insight leads to the \textit{stability theorem} for \v{C}ech persistent homology \cite{cohen-steinerStabilityPersistenceDiagrams2007}.

\subsection{Bifiltrations}
However, the \v{C}ech complex is not sensitive to the density of the point cloud.
To address this, it has become a focus of recent TDA research to introduce a second filtration parameter; see \cite{blumberg_stability_2022} and the survey \cite{botnan_introduction_2022}.
In fact, if $X$ is not a subset, but a multiset (repeating data points in a sample are not unheard of in practice), one would like to adapt the definition of the offset filtration to account for this fact.
This motivates the following definition:
\begin{definition}[{\cite{sheehy12multicover}}]\label{def:MulticoverBifiltration}
    Let $Z$ be a topological space and let $\,\mathcal{ U}$ be a cover (which may contain repeated elements), define the \textit{multicover filtration of} $\mathcal{ U}$ as
    \begin{align*}
        \mathcal{M}(\mathcal{U})\colon ]0,\infty[^{op} & \to \Top,                                                                                                     \\
        \mathcal{M}(\mathcal{U})_m                     & =\{z\in Z \colon z \textnormal{ is contained in at least }m\textnormal{ elements of the cover }\,\mathcal{ U}\}.
    \end{align*}

    If $X$ is a finite subset of an ambient metric space $(Z,d)$, the \textit{multicover bifiltration of} $X$ is, at fixed scale $r$, the multicover filtration induced by the covering given by closed $r$-balls:
    \begin{align*}
        \M(X)\colon ]0,\infty[^{op}\times[0,\infty[ & \to \Top,                                                                             \\
        \M(X)_{m,r}                                 & = \{z\in Z: d(z,x)\leq r \textnormal{ for at least }m\textnormal{ elements }x\in X\}.
    \end{align*}
\end{definition}

\begin{example}\label{example:MulticoverFiltration}
    \begin{figure}
        \centering
        \begin{tikzpicture}
\def\centerarc[#1](#2)(#3:#4:#5)
    { \draw[#1] ($(#2)+({(#5)*cos(#3)},{(#5)*sin(#3)})$) arc (#3:#4:#5); }

    \node at (-3,7) {$\mathcal{M}(X)_{m,r}$};
    \node at (-3,0) {$r=\sqrt{2}/2$};

    \node at (0,7) {$m=1$};
    \node at (3,7) {$m=2$};
    \node at (6,7) {$m=3$};
    \node at (9,7) {$m=4$};

    \node at (0,2) {\rotatebox{90}{$\subseteq$}};
    \node at (3,2) {\rotatebox{90}{$\subseteq$}};
    \node at (6,2) {\rotatebox{90}{$\subseteq$}};
    \node at (9,2) {\rotatebox{90}{$\subseteq$}};

    \draw[fill=blue!50, fill opacity=0.33] (0,0) circle ({sqrt(2)/2});
    \draw[fill=blue!50, fill opacity=0.33] (0,-1) circle ({sqrt(2)/2});
    \draw[fill=blue!50, fill opacity=0.33] (0.5,{sqrt(3)/2}) circle ({sqrt(2)/2});
    \draw[fill=blue!50, fill opacity=0.33] (-0.5,{sqrt(3)/2}) circle ({sqrt(2)/2});
    
    \draw[fill=gray] (0,0) circle (1pt);
    \draw[fill=gray] (0,-1) circle (1pt);
    \draw[fill=gray] (0.5,{sqrt(3)/2}) circle (1pt);
    \draw[fill=gray] (-0.5,{sqrt(3)/2}) circle (1pt);

    \node at (1.5,0) {$\supseteq$};

    \centerarc[fill=blue!50, fill opacity=0.33](3,0)(15:105:sqrt(0.5))
    \centerarc[fill=blue!50, fill opacity=0.33](3,0)(75:165:sqrt(0.5))
    \centerarc[fill=blue!50, fill opacity=0.33](3,0)(225:315:sqrt(0.5))
    \centerarc[fill=blue!50, fill opacity=0.33](3,-1)(45:135:sqrt(0.5))
    
    \centerarc[fill=blue!50, fill opacity=0.33](3.5,{sqrt(3)/2})(135:225:sqrt(0.5))
    \centerarc[fill=blue!50, fill opacity=0.33](3.5,{sqrt(3)/2})(195:285:sqrt(0.5))
    \centerarc[fill=blue!50, fill opacity=0.33](2.5,{sqrt(3)/2})(255:345:sqrt(0.5))
    \centerarc[fill=blue!50, fill opacity=0.33](2.5,{sqrt(3)/2})(-45:45:sqrt(0.5))

    \draw[gray] (3,0) circle (1pt);
    \draw[gray] (3,-1) circle (1pt);
    \draw[gray] (3.5,{sqrt(3)/2}) circle (1pt);
    \draw[gray] (2.5,{sqrt(3)/2}) circle (1pt);

    \node at (4.5,0) {$\supseteq$};
        
    \fill[fill=blue!50, fill opacity=0.33] 
        (6,{(sqrt(3)-1)/2}) -- ({6+(sqrt(3)-1)/4},{(1+sqrt(3))/4}) -- ({6-(sqrt(3)-1)/4},{(1+sqrt(3))/4}) -- (6,{(sqrt(3)-1)/2});
    \centerarc[fill=blue!50, fill opacity=0.33](6,0)(75:105:sqrt(0.5))
    \centerarc[fill=blue!50, fill opacity=0.33](6.5,{sqrt(3)/2})(195:225:sqrt(0.5))
    \centerarc[fill=blue!50, fill opacity=0.33](5.5,{sqrt(3)/2})(-45:-15:sqrt(0.5))

    \draw[gray] (6,0) circle (1pt);
    \draw[gray] (6,-1) circle (1pt);
    \draw[gray] (6.5,{sqrt(3)/2}) circle (1pt);
    \draw[gray] (5.5,{sqrt(3)/2}) circle (1pt);

    \node at (7.5,0) {$\supseteq$};
    \node at (9,0) {$\emptyset$};

    \def\dy{-4.5}

    \node at (-3,{0-\dy}) {$r=1$};

    \draw[fill=blue!50, fill opacity=0.33] (0,{0-\dy}) circle ({sqrt(1)});
    \draw[fill=blue!50, fill opacity=0.33] (0,{-1-\dy}) circle ({sqrt(1)});
    \draw[fill=blue!50, fill opacity=0.33] (0.5,{sqrt(3)/2-\dy}) circle ({sqrt(1)});
    \draw[fill=blue!50, fill opacity=0.33] (-0.5,{sqrt(3)/2-\dy}) circle ({sqrt(1)});
    
    \draw[fill=gray] (0,{0-\dy}) circle (1pt);
    \draw[fill=gray] (0,{-1-\dy}) circle (1pt);
    \draw[fill=gray] (0.5,{sqrt(3)/2-\dy}) circle (1pt);
    \draw[fill=gray] (-0.5,{sqrt(3)/2-\dy}) circle (1pt);

    \node at (1.5,{0-\dy}) {$\supseteq$};

    \centerarc[fill=blue!50, fill opacity=0.33](3,{0-\dy})(0:120:sqrt(1))
    \centerarc[fill=blue!50, fill opacity=0.33](3,{0-\dy})(60:180:sqrt(1))
    \centerarc[fill=blue!50, fill opacity=0.33](3,{0-\dy})(210:330:sqrt(1))
    \centerarc[fill=blue!50, fill opacity=0.33](3,{-1-\dy})(30:150:sqrt(1))
    
    \centerarc[fill=blue!50, fill opacity=0.33](3.5,{sqrt(3)/2-\dy})(120:240:sqrt(1))
    \centerarc[fill=blue!50, fill opacity=0.33](3.5,{sqrt(3)/2-\dy})(180:300:sqrt(1))
    \centerarc[fill=blue!50, fill opacity=0.33](2.5,{sqrt(3)/2-\dy})(240:360:sqrt(1))
    \centerarc[fill=blue!50, fill opacity=0.33](2.5,{sqrt(3)/2-\dy})(-60:60:sqrt(1))

    \draw[fill=gray] (3,{0-\dy}) circle (1pt);
    \draw[fill=gray] (3,{-1-\dy}) circle (1pt);
    \draw[fill=gray] (3.5,{sqrt(3)/2-\dy}) circle (1pt);
    \draw[fill=gray] (2.5,{sqrt(3)/2-\dy}) circle (1pt);

    \node at (4.5,{0-\dy}) {$\supseteq$};
        
    \fill[fill=blue!50, fill opacity=0.33] 
        (6,{(0-\dy}) -- 
        ({6.5},{sqrt(3)/2-\dy}) -- 
        ({5.5},{sqrt(3)/2-\dy}) -- 
        (6,{0-\dy});
    \centerarc[fill=blue!50, fill opacity=0.33](6,{0-\dy})(60:120:sqrt(1))
    \centerarc[fill=blue!50, fill opacity=0.33](6.5,{sqrt(3)/2-\dy})(180:240:sqrt(1))
    \centerarc[fill=blue!50, fill opacity=0.33](5.5,{sqrt(3)/2-\dy})(-60:0:sqrt(1))

    \centerarc[fill=blue!50, fill opacity=0.33](6.5,{sqrt(3)/2-\dy})(240:270:sqrt(1))
    \centerarc[fill=blue!50, fill opacity=0.33](5.5,{sqrt(3)/2-\dy})(-90:-60:sqrt(1))
    \centerarc[fill=blue!50, fill opacity=0.33](6,{-1-\dy})(60:90:sqrt(1))
    \centerarc[fill=blue!50, fill opacity=0.33](6,{-1-\dy})(90:120:sqrt(1))

    \draw[fill=gray] (6,{0-\dy}) circle (1pt);
    \draw[gray] (6,{-1-\dy}) circle (1pt);
    \draw[fill=gray] (6.5,{sqrt(3)/2-\dy}) circle (1pt);
    \draw[fill=gray] (5.5,{sqrt(3)/2-\dy}) circle (1pt);

    \node at (7.5,{0-\dy}) {$\supseteq$};
    \draw[fill=gray] (9,{0-\dy}) circle (1pt);

\end{tikzpicture}
        \caption{
            The multicover filtration of a covering by closed $r$-balls centered a four points in the Euclidean plane.}
        \label{fig:MulticoverFiltration}
    \end{figure}
    Consider $X$ to be the four points $\{(0,0),(0,-1),(\frac{1}{2},\frac{\sqrt{3}}{2}),(-\frac{1}{2},\frac{\sqrt{3}}{2})\}$ in the Euclidean plane depicted in Figure~\ref{fig:MulticoverFiltration} and let us look at its multicover bifiltration.
    In the top row, we fix $r=1$ and in the bottom row, we fix $r=\sqrt{2}/2$. 
    The multicover bifiltration of $X$ restricted to $r$ in the second parameter is equivalently given by the multicover filtration of the covering by closed $r$-balls. 
    At given $m$ it contains those points in $\R^2$ within distance $r$ of at least $m$ points of $X$.
    The leftmost columns shows $m=1$; this restriction of the multicover bifiltration is just the offset filtration.
    Then for $m=2$ in the second column, we get all the points covered by at least two $r$-balls.
    In the third column, when $m=3$, we obtain the threefold intersection of the balls, which consists of all points that are within distance $r$ of three points in $X$.
    Finally, in the rightmost column, we have $m=4$, which is empty for $r=\sqrt{2}/2$ and consists of just a single point for $r=1$.
    Note that the data points themselves do not belong to the multicover bifiltration for some pairs of parameters (this does not happen for the offset filtration).
\end{example}

The key idea to account for multiplicities in $X$ is to consider the associated counting measure $\mu_x = \sum_{x\in X}\delta_x$.
In this context it is useful to recall that a metric space is called \textit{Polish} if it is complete and separable. 
When dealing with measures on some metric space $Z$, we always assume it is Polish and equipped with the Borel $\Sigma$-algebra $\Borel(Z)$. 
A \textit{metric measure space} $(Z,d,\mu)$ consists of a Polish space $(Z,d)$ and a Borel measure $\mu$; if $\mu(Z)=1$ we say $(Z,d,\mu)$ is a \textit{metric probability space}.
Any finite metric space $(X,d)$ comes with two canonical measures, the \textit{counting measure} $\mu_X=\sum_{x\in X}\delta_x$ we have already seen, and the \textit{empirical probability measure} $\nu_X = \frac{1}{|X|}\sum_{x\in X}\delta_x$.

If a point $x\in X$ now repeats $M_x$ times, we can simply consider the measure $\sum_{x\in X}M_x \delta_x$.
With this in mind, we recast the multicover bifiltration of a finite set of points $X$ as
\begin{align*}
    \M(X)_{m,r}
     & = \{z\in Z: d(z,x)\leq r \textnormal{ for at least }m\textnormal{ elements }x\in X\} \\
     & = \{z\in Z: \vert X \cap \overline{B}_r(z)\vert \geq m\}                             \\
     & = \{z\in Z: \mu_X(\overline{B}_r(z)) \geq m\}.
\end{align*}

Thus, the multicover bifiltration naturally generalizes to arbitrary measures as follows:

\begin{definition}[{\cite{blumberg_stability_2022}}]\label{def:MeasureBifiltration}
    Let $(Z,d)$ be a Polish space and $\mu$ be a Borel measure on it.
    The \textit{measure bifiltration} is
    \begin{align*}
        \mathcal{B}(\mu)\colon ]0,\infty[^{op}\times[0,\infty[ & \to \Top,                                         \\
        \mathcal{B}(\mu)_{m,r}                                 & = \{z\in Z \colon \mu(\overline{B}_r(z))\geq m\}.
    \end{align*}
\end{definition}
In particular, this bifiltration handles points with multiplicities as described above.
This general formulation is useful to prove stability and robustness results (as we shall see later on), but not for computations.
Just like the offset filtration admits a combinatorial model in form of the \v{C}ech filtration, the multicover filtration is weakly equivalent to the so-called subdivision \v{C}ech bifiltration \cite[Theorem 3.3 (i)]{blumberg_stability_2022}.
This is defined as follows:
\begin{definition}[{\cite{sheehy12multicover}}]\label{def:SubdivisionBifiltration}
    Let $K$ be a simplicial complex and denote its barycentric subdivision by $\Sd(K)$.
    A $k$-simplex in $\Sd(K)$ is given by an ascending chain $\sigma_0 \subsetneq \ldots \subsetneq \sigma_k$ (called a \textit{flag}) of simplices in $K$.
    The \textit{subdivision filtration} $\mathcal{S}(K)$ at index $m$  is given by the complex whose simplices are flags in which the minimal dimension is at least $m-1$,
    \begin{align*}
        \mathcal{S}(K)\colon]0,\infty[^{op} & \to \Simp,                                                                                          \\
        \mathcal{S}(K)_m                    & =\{ (\sigma_0\subsetneq\ldots\subsetneq\sigma_k) \colon \dim(\sigma_0) \geq m-1\} \subseteq \Sd(K).
    \end{align*}

    Let $(Z,d)$ be a Polish space and $X$ a finite subset.
    The \textit{subdivision \v{C}ech bifiltration} is
    \begin{align*}
        \SCech(X)\colon ]0,\infty[^{op}\times[0,\infty[ & \to \Simp,                     \\
        \SCech(X)_{m,r}                                 & = \mathcal{S}(\Cech(X)_{r})_m.
    \end{align*}
    Let $X$ be a non-empty finite metric space.
    Its \textit{subdivision Rips bifiltration} is
    \begin{align*}
        \SRips(X)\colon ]0,\infty[^{op}\times[0,\infty[ & \to \Simp,                   \\
        \SRips(X)_{m,r}                                 & = \mathcal{S}(\VR(X)_{r})_m.
    \end{align*}
\end{definition}
The equivalence between multicover and subdivision bifiltrations for certain\footnote{Such metric spaces are called \textit{good} in \cite{blumberg_stability_2022}; for instance, the Euclidean space $\R^d$ is good.} metric spaces is established by the following theorem, which we state in abbreviated form -- see \cite[section 4]{blumberg_stability_2022} for a more thorough discussion.
Our Corollary \ref{cor:SimplicialMulticover} can be thought of as a close analogue of this result.

\begin{theorem}[{Multicover Nerve Theorem, \cite{sheehy12multicover}, \cite{cavannaWhenWhyTopological2017},\cite[{Theorem 4.12 and Remark 4.13}]{blumberg_stability_2022}; see also \cite{Bauer_unified_2023}}]\label{thm:MulticoverNerveTheorem}
    Given a poset $T$ and a filtration $F\colon T\to \Top$ of compactly generated spaces, suppose we have a set $\mathcal{U}$ of functors $T\to \Top$ such that
    \begin{enumerate}[i)]
        \item for every $t\in T$, the set $\{U_t \colon U \in \mathcal{U}\}$ is a closed cover of $F_t$ such that every finite non-empty intersection is weakly homotopy equivalent to a point and satisfying the conditions of \cite[Theorem 5.9.1 b)]{Bauer_unified_2023},
        \item for every $U \in \mathcal{U}$ and every $s\leq t \in T$, the map $U_{s\leq t}$ is the restriction of $F_{s\leq t}$ to $U_s$.
    \end{enumerate}
    Then we have a weak equivalence of filtrations $\mathcal{M}(\mathcal{U})\simeq\vert\mathcal{S}(\Nerve{\mathcal{U}})\vert$, where
    \begin{align*}
        \mathcal{M}(\mathcal{U}), \vert\mathcal{S}(\Nerve{\mathcal{U}})\vert\colon]0,\infty[^{op}\times T 
        & \to \Top,\\
        \mathcal{M}(\mathcal{U})_{m,t}
        & = \mathcal{M}(\{U_t\colon U\in \mathcal{U}\})_m,\\
        \vert\mathcal{S}(\Nerve{\mathcal{U}})\vert_{m,t}
        &= \vert\mathcal{S}(\Nerve{\{U_t\colon U\in\mathcal{U}\}})_m\vert.
    \end{align*}
\end{theorem}
As an example, note that for a finite set $X\subseteq \R^d$, the offset filtration $\mathcal{O}(X)$ can be taken as $F$ in the theorem.
In this setting, the set of functors giving covers is indexed by the poset $[0,\infty[$ and we have $\mathcal{U} = \{\overline{B}_\bullet(x)\}_{x\in X}$, where
\[
    \overline{B}_\bullet(x) \colon [0,\infty[ \to \Top,\qquad r\mapsto \overline{B}_r(x).
\]
Let us illustrate this by continuing Example~\ref{example:MulticoverFiltration}.
\begin{example}\label{example:Bifiltrations}
    \begin{figure}
        \centering
        \begin{tikzpicture}
\def\centerarc[#1](#2)(#3:#4:#5)
    { \draw[#1] ($(#2)+({(#5)*cos(#3)},{(#5)*sin(#3)})$) arc (#3:#4:#5); }

    \node at (-2,0) {$\mathcal{M}(X)_{m,r}$};
    \node at (-2,-3.5) {$\SCech(X)_{m,r}$};

    \node at (0,2) {$m=1$};
    \node at (3,2) {$m=2$};
    \node at (6,2) {$m=3$};

    \draw[fill=blue!50, fill opacity=0.33] (0,0) circle ({sqrt(2)/2});
    \draw[fill=blue!50, fill opacity=0.33] (0,-1) circle ({sqrt(2)/2});
    \draw[fill=blue!50, fill opacity=0.33] (0.5,{sqrt(3)/2}) circle ({sqrt(2)/2});
    \draw[fill=blue!50, fill opacity=0.33] (-0.5,{sqrt(3)/2}) circle ({sqrt(2)/2});
    
    \draw[fill=gray] (0,0) circle (1pt);
    \draw[fill=gray] (0,-1) circle (1pt);
    \draw[fill=gray] (0.5,{sqrt(3)/2}) circle (1pt);
    \draw[fill=gray] (-0.5,{sqrt(3)/2}) circle (1pt);

    \node at (1.5,0) {$\supseteq$};

    \centerarc[fill=blue!50, fill opacity=0.33](3,0)(15:105:sqrt(0.5))
    \centerarc[fill=blue!50, fill opacity=0.33](3,0)(75:165:sqrt(0.5))
    \centerarc[fill=blue!50, fill opacity=0.33](3,0)(225:315:sqrt(0.5))
    \centerarc[fill=blue!50, fill opacity=0.33](3,-1)(45:135:sqrt(0.5))
    
    \centerarc[fill=blue!50, fill opacity=0.33](3.5,{sqrt(3)/2})(135:225:sqrt(0.5))
    \centerarc[fill=blue!50, fill opacity=0.33](3.5,{sqrt(3)/2})(195:285:sqrt(0.5))
    \centerarc[fill=blue!50, fill opacity=0.33](2.5,{sqrt(3)/2})(255:345:sqrt(0.5))
    \centerarc[fill=blue!50, fill opacity=0.33](2.5,{sqrt(3)/2})(-45:45:sqrt(0.5))

    \draw[gray] (3,0) circle (1pt);
    \draw[gray] (3,-1) circle (1pt);
    \draw[gray] (3.5,{sqrt(3)/2}) circle (1pt);
    \draw[gray] (2.5,{sqrt(3)/2}) circle (1pt);

    \node at (4.5,0) {$\supseteq$};
        
    \fill[fill=blue!50, fill opacity=0.33] 
        (6,{(sqrt(3)-1)/2}) -- ({6+(sqrt(3)-1)/4},{(1+sqrt(3))/4}) -- ({6-(sqrt(3)-1)/4},{(1+sqrt(3))/4}) -- (6,{(sqrt(3)-1)/2});
    \centerarc[fill=blue!50, fill opacity=0.33](6,0)(75:105:sqrt(0.5))
    \centerarc[fill=blue!50, fill opacity=0.33](6.5,{sqrt(3)/2})(195:225:sqrt(0.5))
    \centerarc[fill=blue!50, fill opacity=0.33](5.5,{sqrt(3)/2})(-45:-15:sqrt(0.5))

    \draw[gray] (6,0) circle (1pt);
    \draw[gray] (6,-1) circle (1pt);
    \draw[gray] (6.5,{sqrt(3)/2}) circle (1pt);
    \draw[gray] (5.5,{sqrt(3)/2}) circle (1pt);

    \node at (0,-2) {\rotatebox{90}{$\simeq$}};
    \node at (3,-2) {\rotatebox{90}{$\simeq$}};
    \node at (6,-2) {\rotatebox{90}{$\simeq$}};

    \fill[fill=gray!50] (0,-3.5) --(0.5,{sqrt(3)/2-3.5})--(-0.5,{sqrt(3)/2-3.5}) -- (0,-3.5);
    \draw[fill] (0,-4.5) circle (1pt)
              --(0,-4) circle (1pt)
              --(0,-3.5) circle (1pt)
              --(0.5,{sqrt(3)/2-3.5}) circle (1pt)
              --(-0.5,{sqrt(3)/2-3.5}) circle (1pt);
    \draw[fill] (0.5,{sqrt(3)/2-3.5})
              --  (-0.25,{(sqrt(3)/2-7)*0.5})circle (1pt)
              --(-0.5,{sqrt(3)/2-3.5}) circle (1pt)
              --(0.25,{(sqrt(3)/2-7)*0.5})circle (1pt);
    \draw[fill] (0,{sqrt(3)/2-3.5}) circle (1pt)
              --(0,-3.5);
    \draw (0,-3.5)--(-0.25,{(sqrt(3)/2-7)*0.5});
    \draw[fill] (0,{-3.5+sqrt(3)/3}) circle (1pt);

    \node at (1.5,-3.5) {$\supseteq$};

    \draw[fill] (3,-4) circle (1pt);
    \draw[fill] (3,{sqrt(3)/2-3.5}) circle (1pt)
        -- (3,{-3.5+sqrt(3)/3}) circle (1pt)
        --(3.25,{(sqrt(3)/2-7)*0.5})circle (1pt);
    \draw[fill](3,{-3.5+sqrt(3)/3}) -- (2.75,{(sqrt(3)/2-7)*0.5})circle (1pt);

    \node at (4.5,-3.5) {$\supseteq$};
    \draw[fill] (6,{-3.5+sqrt(3)/3}) circle (1pt);
\end{tikzpicture}
        \caption{Four points $X$ in the Euclidean plane with a portion of its multicover bifiltration and the equivalent subdivision \v{C}ech bifiltration for a fixed value of $r$.}
        \label{fig:BifiltrationIllustration}
    \end{figure}
    Recall the setting of Example~\ref{example:MulticoverFiltration}, in which we described the multicover filtration of a set of four points in the Euclidean plane.
    We repeat a part of that picture in the top half of Figure~\ref{fig:BifiltrationIllustration}.
    A combintorial model is given by the subdivision \v{C}ech bifiltration, which is shown in the bottom row.
\end{example}

However, due to the appearance of barycentric subdivisions, these complexes are usually intractable for use in computations.
Instead, one often considers the following subcomplexes of the Rips filtration, although there are other options like the rhomboid bifiltration \cite{edelsbrunnerMultiCoverPersistenceEuclidean2021,corbetComputingMulticoverBifiltration2023}, which is equivalent to the multicover bifiltration in Euclidean space, but computationally much less expensive.
\begin{definition}[{\cite{lesnickInteractiveVisualization2D2015a}}]
    Let $(Z,d)$ be a Polish space and $\mu$ be a Borel probability measure on it.
    The \textit{degree Rips bifiltration} is
    \begin{align*}
        \DRips(Z,\mu)\colon ]0,\infty[^{op}\times[0,\infty[ & \to \Simp,                       \\
        \DRips(Z,\mu)_{m,r}                                 & = \VR(\mathcal{B}(\mu)_{m,r})_r.
    \end{align*}
\end{definition}
That is, at each stage $(m,r)$, we evaluate the measure bifiltration and take the resulting metric subspace as an input for the Rips complex.
The name comes from the fact that for $\mu=\mu_X=\sum_{x\in X}\delta_x$ being the counting measure of a finite metric space $X$, the set $\mathcal{B}(\mu_X)_{m,r}$ is precisely the set of vertices in $\VR(X)_r$ of degree $\geq m-1$.

Our next aim is to describe the stability of bifiltrations in analogy to the stability of one-parameter filtrations.
Recall that the Hausdorff distance is appropriate to measure distances between point clouds in the one-parameter setting.
In the two-parameter setting, we are going to use the Prokhorov metric (Definition~\ref{def:Prokhorov}), which can be thought of as a density-sensitive analogue of the Hausdorff distance.
Therefore, we wish to regard the input data as a probability measure.
While the measure bifiltration is built to handle general measures, the other bifiltrations are defined more combinatorially and hence need to be modified.
To motivate this modification, recall that the measure bifiltration of a counting measure of a finite set of points $X$ in $\R^d$ is the same as its multicover bifiltration, $\mathcal{M}(X) = \mathcal{B}(\mu_X)$.
Recall furthermore that the empirical probability measure associated to $X$ is obtained by normalizing the counting measure, $\nu_X = \frac{1}{|X|}\mu_X$.
Thus, for any $m>0$ and any Borel set $A\subset \R^d$, we have $\nu_X(A)\geq m \Leftrightarrow \mu_X(A) \geq |X|m$ and thus $\mathcal{B}(\nu_X)_{m,r} = \mathcal{M}(X)_{|X|m,r}$ for any $r\geq 0$.
In this vein, we introduce normalized bifiltrations.
\begin{definition}\label{def:NormalizedBifiltrations}
    For a non-empty finite metric space $(X,d)$, define the following normalized bifiltrations
    \begin{align*}
        \mathcal{M}^n(X) \colon ]0,\infty[^{op}\times [0,\infty[&\to \Top,\\
        \SRips^n(X)\colon ]0,\infty[^{op}\times [0,\infty[&\to \Simp.
    \end{align*}
    \begin{itemize}
        \item The \textit{normalized multicover bifiltration} is given by $\mathcal{M}^n(X)_{m,r} := \mathcal{M}(X)_{|X|m,r}$,
        \item the \textit{normalized subdivision Rips bifiltration} is given by $\SRips^n(X)_{m,r} := \SRips(X)_{|X|m,r}$,
    \end{itemize}
    Analogously, for a finite simplicial complex $K$, its \textit{normalized subdivision filtration} is given by $\mathcal{S}^n(K)_{m} = \mathcal{S}(K)_{|K_0|m}$.
\end{definition}

\subsubsection{Metrics and Stability}

The stability of these constructions is due to \cite{blumberg_stability_2022} and \cite{rolle2023stable}.
Just like the Hausdorff distance compares subsets of a metric space, we need a way to compare probability measures; a common choice in the context of TDA is the Prokhorov metric:
\begin{definition}\label{def:Prokhorov}
Let $\mu$ and $\eta$ be probability measures on a common metric space $(X,d)$.
    The \textit{Prokhorov metric} between $\mu$ and $\eta$ is
    \[
    d_{Pr}({\mu,\eta}) = \inf\{\eps> 0 \colon \mu(A)\leq \eta(A^\eps)+\eps \text{ and } \eta(A)\leq \mu(A^\eps)+\eps \textnormal{ for all closed }A \subseteq X \},
    \]
    where $A^\eps = \{x\in X \colon \exists a\in A\colon d(a,x)\leq \eps\}.$
\end{definition}
By combining Hausdorff and Prokhorov distances, we obtain the following:
\begin{definition}\label{def:GromovHausforffProkhorov}
    Let $(X_1,d_1,\mu_1)$ and $(X_2,d_2,\mu_2)$ be two metric probability spaces.
    \begin{enumerate}
    \item Their \textit{Gromov-Prokhorov distance}~\cite{grevenConvergenceDistributionRandom2009} is 
    \[
            d_{GPr}(\mu_1,\mu_2) = \inf\limits_{X_1 \xrightarrow[]{\varphi}Z\xleftarrow[]{\psi}X_2}  d_{Pr}(\varphi_\#(\mu_1),\psi_\#(\mu_2))
        \]
    \item Their \textit{Gromov-Hausdorff-Prokhorov distance}~\cite{abrahamNoteGromovHausdorffProkhorovDistance2013} is 
        \begin{align*}
            d_{GHPr}((X_1,d_1,\mu_1),&(X_2,d_2,\mu_2)) \\
            &= \inf\limits_{X_1 \xrightarrow[]{\varphi}Z\xleftarrow[]{\psi}X_2} \max\{d_H(\varphi(X_1),\psi(X_2)), d_{Pr}(\varphi_\#(\mu_1),\psi_\#(\mu_2))\}, 
        \end{align*}
    \end{enumerate}
        where both times the infimum ranges over all isometric embeddings into a common Polish metric space $(Z,d)$, in which the Hausdorff and Prokhorov distances are evaluated.
\end{definition}

Let us recall the definition of interleavings for general posets $T$; for a more comprehensive perspective, refer to \cite[Section 2.5]{blumberg_stability_2022} and \cite[Section 6.1]{botnan_introduction_2022}.
Usual choices for $T$ include $\N,\Z,\R, [0,\infty[$ with partial order being ``$\leq$'', their opposites and cartesian products thereof.
For a functor $F\colon T\to \mathbf{C}$ we follow the convention of writing $F_t:=F(t)$ and $F_{s\leq t}$ for the map induced by the unique morphism for $s\leq t \in T$.
\begin{definition}\label{def:ForwardShift}
    Let $T$ be any of the posets $[0,\infty[, \R^{op}, \R^{op}\times [0,\infty[$.
    A poset automorphism $\alpha\colon T\to T$ is called a \textit{forward shift} if $t\leq \alpha(t)$ for all $t\in T$.
\end{definition}

\begin{definition}\label{def:Interleaving}
    Let $(T,T')$ be either of the following pairs of posets:
    \[
        ([0,\infty[,[0,\infty[); \qquad (\R^{op}\times [0,\infty[, ]0,\infty[^{op} \times [0,\infty[).
    \]
    Let $\alpha, \beta\colon T \to T$ be forward shifts.
    The $(\alpha,\beta)$\textit{-interleaving category} of $T'$ is $\mathbf{I}(T',\alpha,\beta)$ and has $T'\times\{0,1\}$ as objects and morphisms $(r,i)\to (s,j)$ if and only if either
    \begin{itemize}
        \item $i=j$ and $r\leq s$,
        \item $i=0$, $j=1$ and $\alpha(r)\leq s$,
        \item $i=1$, $j=0$ and $\beta(r)\leq s$.
    \end{itemize}
    Composition in $\mathbf{I}(T',\alpha,\beta)$ is defined by the requirement that between any two objects, there is at most one morphisms (i.e. the category is thin).
    We can include $T'$ in $\mathbf{I}(T',\alpha,\beta)$ in two ways, namely $E_i \colon T'\to T'\times\{i\}$ via the identity on $T'$ for $i\in \{0,1\}$.
    Given two functors $F,G\colon T'\to \mathbf{C}$, an $(\alpha,\beta)$-interleaving is a functor $Z\colon \mathbf{I}(T',\alpha,\beta)\to \mathbf{C}$ such that
    \[
        F = Z\circ E_0 \textnormal{ and } G = Z \circ E_1.
    \]
    
    For short, we say that an $(\alpha,\alpha)$-interleaving is an $\alpha$-interleaving, and for $\delta>0$ we say we have a $\delta$ interleaving if $\alpha(t)=t+\delta$ (in the one-parameter case) or $\alpha((m,r))=(m-\delta,r+\delta)$ (in the two-parameter case).
    
    The \textit{interleaving distance} of two functors $F,G$ as above is
    \[
        d_I(F,G) = \inf\{\delta>0 \colon F,G \textnormal{ are } \delta\textnormal{-interleaved}\}.
    \]

    Moreover,  $d_{HI}$ denotes the \textit{homotopy interleaving distance} between $\Top$-valued functors,
    \[
        d_{HI}(F,G) = \inf\{\delta > 0 \colon \exists F'\simeq F,\;  G'\simeq G \textnormal{ such that } F', G' \textnormal{ are }\delta-\textnormal{interleaved}\},
    \]
    where $\simeq$ denotes a weak equivalence of functors $T\to \Top$, cf. Definition \ref{def:Filtration}.
\end{definition}
The intuition is that the interleaving distance thus measures, how far away $F$ and $G$ are from being isomorphic, with a $0$-interleaving yielding an isomorphism.

Now we are able to phrase stability theorems for two-parameter filtrations.
\begin{theorem}[{\cite[Theorem 1.6 i)]{blumberg_stability_2022}}]\label{thm:StabilityMeasureBifiltration}
    For any two finite subsets $X_1, X_2$ of a common metric space $Z$ with associated empirical probability measures $\nu_1, \nu_2$ on $\Borel(Z)$ we have
    \[
        d_I(\mathcal{M}(X_1), \mathcal{M}(X_2)) \leq d_{Pr}(\nu_1,\nu_2).
    \]
\end{theorem}
\begin{theorem}[{\cite[Theorem 1.6 iii)]{blumberg_stability_2022}}]\label{thm:SubdivisionRipsRobustness}
    Let $X_1,X_2$ be two non-empty finite metric spaces endowed with their empirical probability measures $\nu_1, \nu_2$.
    Then\footnote{Recall that the definitions of the Vietoris--Rips complex of \cite{blumberg_stability_2022} and ours differ by a factor of two.}
    \[
        d_{HI}(\SRips^n(X_1)_{\bullet,2\bullet},\SRips^n(X_2)_{\bullet, 2\bullet}) \leq d_{GPr}(\nu_1,\nu_2).
    \]
\end{theorem}
\begin{theorem}[{\cite[Theorem 6.5.1]{scoccolaLocallyPersistentCategories2020}, \cite[Theorem 1.7]{blumberg_stability_2022}}]\label{thm:DegRipsStability}
    For any metric probability spaces $(X_1,d_1,\mu_1)$, $(X_2,d_2,\mu_2)$, we have
    \[
        d_{HI} (\DRips(X_1,d_1,\mu_1), \DRips(X_2,d_2,\mu_2)) \leq d_{GHPr}((X_1,d_1,\mu_1), (X_2,d_2,\mu_2) );
    \]
    moreover, for any $\delta > d_{GPr}(\mu_1, \mu_2)$, we have a homotopy-interleaving with respect to the forward-shift $(m,r)\to (m-\delta,3r+\delta)$
\end{theorem}
Since the bound in Theorem \ref{thm:SubdivisionRipsRobustness} is only with respect to Gromov-Prokhorov (without Hausdorff -- although the metric structure is implicit as one takes the infimum over isometric embeddings), we can interpret it as ascertaining that subdivision-Rips is \textit{robust}.
The degree-Rips bifiltration only satisfies a weaker robustness result (with a multiplicative factor of 3).
Moreover, we can interpret the appearance of the Hausdorff distance as $\DRips$ being more easily affected by metric perturbations.
This presents a trade-off between computability and robustness; Theorem \ref{thm:MNeighborRobust} and Corollaries \ref{cor:RobustnessFixedLandmarks} and \ref{cor:LandmarksApproximateMulticover} can be viewed as a step twoards addressing this issue.

This finishes the preliminaries on density-sensitive bifiltrations, which are one  of two major prerequisites for the present article.
\subsection{Dowker's Complexes and Theorem}
The other prerequisite is the construction of Dowker complexes; before we turn to it, we first introduce the category of relations.

\begin{definition}\label{def:Relation}
    Given two sets $X,Y$, a \textit{relation} $R\subseteq X\times Y$ is a subset of the product.
    The category $\Rel$ has triples $(X,Y,R)$ as objects, where $R\subseteq X\times Y$.
    Its morphisms are $f=(f_X,f_Y)\colon (X,Y,R)\to(X',Y',R')$, where $f_X\colon X\to X'$ and $f_Y\colon Y \to Y'$ are maps such that $(x,y)\in R$ implies $(f_X(x),f_Y(y))\in R'$.
    Composition is defined component-wise.
\end{definition} 

We will usually identify a relation with its indicator matrix, which is a binary matrix with row labels given by $X$ and column labels $Y$.
The entry at $(x,y)$ is $1$ if $(x,y)\in R$ and $0$ otherwise.

Dowker's seminal work introduced a construction of a simplicial complex associated to any relation.

\begin{definition}[\cite{dowker_homology_1952}]\label{def:DowkerComplex}
    Let $X,Y$ be sets and $R\subseteq X\times Y$ a relation.
    The \textit{Dowker complex} of the relation is the abstract simplicial complex $\Dow(X,Y,R)$ whose simplices are the nonempty finite subsets $\sigma\subseteq X$ that satisfy
    \[
         \exists y\in Y \colon \sigma \times \{y\} \subseteq R.
    \]
    If $\sigma\times \{y\}\subseteq R$, we say $y$ is a \textit{witness} of (or: \textit{witnesses}) $\sigma$.
\end{definition}

\begin{definition}\label{def:TransposeRelation}
    Let $R\subseteq X\times Y$ be a relation; we denote by $R^\top\subseteq Y\times X$ its \textit{transpose}, that is,
    \[
        (y,x)\in R^\top \Leftrightarrow (x,y)\in R.
    \]
    If $f=(f_X,f_Y) \colon (X,Y,R) \to (X',Y',R')$ is a map, we get a \textit{transposed map} $f^\top = (f_Y,f_X) \colon (Y,X,R^\top) \to (Y',X',(R')^\top)$.
\end{definition}
Taking the transpose of the indicator matrix of a relation $R$, we obtain the corresponding matrix representing the transpose relation.

The following theorem is originally due to Dowker \cite{dowker_homology_1952}, who proved it only in terms of homology equivalences; the version here pertaining to homotopy equivalences is due to Björner \cite[Theorem 10.9]{bjornerTopologicalMethods1996}.
\begin{theorem}[Dowker duality]\label{thm:DowkerDuality}
    Let $R\subseteq X\times Y$ be a relation and denote by $R^\top\subseteq Y\times X$ its transpose.
    Then we have a homotopy equivalence
    \[
        \vert\Dow(X,Y,R)\vert \simeq \vert\Dow(Y,X, R^\top)\vert.
    \]
\end{theorem}

\begin{example}\label{example:DowkerDuality}
    \begin{figure}
        \centering
        \begin{tikzpicture}

\node[] at (-1,0) {$%
            \bordermatrix{X\setminus Y & \alpha & \beta & \gamma & \delta & \eps \cr%
                 a & 1 & 1 & 0 & 0 & 1\cr%
                 b & 1 & 0 & 0 & 1 & 1\cr%
                 c & 1 & 0 & 1 & 1 & 0\cr%
                 d & 0 & 1 & 1 & 0 & 0}%
            $};%

        \node[] at (-1,-2) {$R$};

                \draw[] (4,0) -- (3.5,-0.866) node[circle, fill=blue!20] {$d$} -- (3,0);
                \draw[fill=gray!50] (4,0) node[circle, fill=blue!20] {$a$} -- (3.5,0.866) node[circle, fill=blue!20](b) {$b$} -- (3,0) node[circle, fill=blue!20](c) {$c$} -- (4,0);

            

        \node at (3.5,-2) {$\Dow(X,Y,R)$};

                \node at (5.5,0) {{$\simeq$}};

                \draw (7,0) to [out=270, in=270] (9,0);
                \draw[fill=gray!50] (8,0)  -- (7.5,0.866)  -- (8.5,0.866) -- (8,0);
                \draw[fill=gray!50] (8,0)  -- (8.5,0.866) node[circle, fill=blue!20] {$\eps$} -- (9,0) node[circle, fill=blue!20] {$\beta$} -- (8,0);
                \draw[fill=gray!50] (7,0) node[circle, fill=blue!20] {$\gamma$} -- (7.5,0.866) node[circle, fill=blue!20] {$\delta$} -- (8,0) node[circle, fill=blue!20] {$\alpha$} -- (7,0);
        \node at (8,-2) {$\Dow(Y,X,R^T)$};
\end{tikzpicture}
        \caption{The indicator matrix representing a binary relation $R$ (left); its Dowker complex $\Dow(X,Y,R)$, whose vertices are the row labels (middle); and the Dowker complex of the dual relation $\Dow(Y,X,R^\top)$, which has the same homotopy type.}
        \label{fig:DowkerDualityExample}
    \end{figure}
    Consider $X=\{a,b,c,d\}$, $Y=\{\alpha,\beta,\gamma,\delta,\eps\}$ and $R\subseteq X\times Y$, which we represent as a binary matrix $R\in \{0,1\}^{X\times Y}$ as indicated in Figure \ref{fig:DowkerDualityExample}, left panel.
    The definition of the Dowker complex unfolds as follows:
    We build a simplicial complex on the vertex set $X$, in which we add a simplex $\sigma \subseteq X$ if there is a column $y\in Y$ which contains $\sigma$; for instance, we introduce the simplex $\{a,b,c\}$ as it is contained in column $\alpha$. 
    This Dowker complex $\Dow(X,Y,R)$ is shown in the middle panel of Figure \ref{fig:DowkerDualityExample}.
    On the other hand, the Dowker complex of the transpose relation, shown  on the right in Figure \ref{fig:DowkerDualityExample}, has vertices $Y$.
    For instance we introduce the simplex $\{\alpha,\delta\}$ because it is contained in row $b$; it is also in row $c$, but in no other row.
    In other words, the rows $b$ and $c$ are witnesses of the simplex $\{\alpha,\delta\}$.
    Observe that $\Dow(X,Y,R)$ and $\Dow(Y,X,R^\top)$ are homotopy equivalent.
\end{example}

It has been variously realized~\cite{bjornerTopologicalMethods1996,chowdhury_functorial_2018} that Dowker duality is equivalent to the nerve theorem in the setting of simplicial complexes covered by simplices.
We recall the formulation for general subcomplexes of Bauer et al.~\cite{Bauer_unified_2023}
\begin{theorem}[{\cite[Theorem 4.8]{Bauer_unified_2023}}]\label{thm:SimplicialNerve}
    Let $K$ be a simplicial complex and let ${\cal A}=(K_i\subseteq K)_{i\in I}$
be a good cover of K by subcomplexes. Then the natural maps 
\[    \varrho_S: \Blowup{|A|}\to |K| \qquad \textrm{and}\qquad \varrho_N:   \Blowup{|A|}\to \Nerve {|\cal A|}              \]   
are homotopy equivalences.
\end{theorem}
The explicit description of $\Blowup{|A|}$ is not needed to follow the arguments presented in the proof of Theorem~\ref{thm:DowkerDualityTotalWeight} below.

The study of functorial aspects of nerve theorems was motivated by persistent homology as an invariant of filtered complexes.
Similarly, recently, functorial extensions of Dowker duality have been established by \cite{chowdhury_functorial_2018}, then later \cite{virk_rips_2021} and recently \cite{brunRectangleComplexRelation2023}.

\begin{proposition}[{\cite[Theorem 5.2]{brunRectangleComplexRelation2023}}]\label{prop:FunctorialDowker}
    Dowker complexes and Dowker Duality are functorial in the following sense:
    Any morphism of relations $f=(f_X,f_Y)\colon (X,Y,R)\to (X',Y',R')$ induces a simplicial map $\Dow(f)\colon \Dow(X,Y,R)\to \Dow(X',Y',R')$; these assemble into a functor $\Dow\colon \Rel \to \Simp$.
    In addition, one can choose homotopy equivalences 
    \[
    \Psi_R\colon \vert\Dow(X,Y,R)\vert\to \vert\Dow(Y,X,R^\top)\vert,\qquad \Psi_{R'}\colon \vert\Dow(X',Y',R')\vert\to \vert\Dow(Y',X',(R')^\top)\vert
    \]
    such that the following diagram commutes up to homotopy:
    \[
        \begin{tikzcd}
            \vert\Dow(X,Y,R)\vert \arrow{r} {\Psi_R} \arrow{d}{\vert\Dow(f)\vert} &
            \vert\Dow(Y,X,R^\top)\vert\arrow{d} {\vert\Dow(f^\top)\vert}\\
            \vert\Dow(X',Y',R')\vert\arrow{r} {\Psi_{R'}} &
            \vert\Dow(Y',X',(R')^\top)\vert
        \end{tikzcd}
    \]
\end{proposition}

In particular, a filtration of relations gives rise to a filtration of Dowker complexes in a way compatible with Dowker duality.
The focus for us will be on relations that arise as sublevel sets of some function, i.e.
\[
    R_r = \{(x,y)\colon \Lambda(x,y)\leq r\}, \textnormal{ where } \Lambda\colon X\times Y \to [0,\infty[.
\]
The most important example is $X,Y$ being subsets of some metric space $(Z,d)$ and $\Lambda=d\vert_{X\times Y}$ being the restriction of the metric.
\begin{example}\label{example:DowkerCech}
\begin{figure}
    \centering
    \resizebox{0.5\linewidth}{!}{\definecolor{zzttqq}{rgb}{0.6,0.2,0}
\definecolor{ududff}{rgb}{0.30196078431372547,0.30196078431372547,1}
\begin{tikzpicture}[line cap=round,line join=round,>=triangle 45,x=1cm,y=1cm]
\clip(-10.069824766802421,-5.413622358441739) rectangle (9.324710750579841,6.06190426365049);
\draw [line width=2pt,fill=black,fill opacity=0.2] (-4.3455,1.80544) circle (2.1cm);
\draw [line width=2pt,fill=black,fill opacity=0.2] (-4.2993772823371055,-2.3460119162021873) circle (2.1cm);
\draw [line width=2pt,fill=black,fill opacity=0.2] (-1.2396040745723542,-0.28566211800883773) circle (2.1cm);
\draw [line width=2pt,fill=black,fill opacity=0.2] (2.404447434172601,0.34474341725927676) circle (2.1cm);
\draw [line width=2pt,fill=black,fill opacity=0.2] (3.867855231045659,2.3390353451257835) circle (2.1cm);
\draw [line width=2pt,fill=black,fill opacity=0.2] (3.5685589684938592,-1.722842503791486) circle (2.1cm);
\fill[line width=2pt,color=zzttqq,fill=zzttqq,fill opacity=0.10000000149011612] (3.867855231045659,2.3390353451257835) -- (3.5685589684938592,-1.722842503791486) -- (2.404447434172601,0.34474341725927676) -- cycle;
\draw [line width=2pt,color=zzttqq] (-4.3455,1.80544)-- (-4.2993772823371055,-2.3460119162021873);
\draw [line width=2pt,color=zzttqq] (-4.2993772823371055,-2.3460119162021873)-- (-1.2396040745723542,-0.28566211800883773);
\draw [line width=2pt,color=zzttqq] (-4.3455,1.80544)-- (-1.2396040745723542,-0.28566211800883773);
\draw [line width=2pt,color=zzttqq] (-1.2396040745723542,-0.28566211800883773)-- (2.404447434172601,0.34474341725927676);
\draw [line width=2pt,color=zzttqq] (2.404447434172601,0.34474341725927676)-- (3.867855231045659,2.3390353451257835);
\draw [line width=2pt,color=zzttqq] (3.867855231045659,2.3390353451257835)-- (3.5685589684938592,-1.722842503791486);
\draw [line width=2pt,color=zzttqq] (3.5685589684938592,-1.722842503791486)-- (2.404447434172601,0.34474341725927676);
\draw [line width=2pt,color=zzttqq] (3.867855231045659,2.3390353451257835)-- (3.5685589684938592,-1.722842503791486);
\draw [line width=2pt,color=zzttqq] (3.5685589684938592,-1.722842503791486)-- (2.404447434172601,0.34474341725927676);
\draw [line width=2pt,color=zzttqq] (2.404447434172601,0.34474341725927676)-- (3.867855231045659,2.3390353451257835);
\begin{scriptsize}
\draw [fill=ududff] (-4.3455,1.80544) circle (2.5pt);
\draw [fill=ududff] (-4.2993772823371055,-2.3460119162021873) circle (2.5pt);
\draw [fill=ududff] (-1.2396040745723542,-0.28566211800883773) circle (2.5pt);
\draw [fill=ududff] (2.404447434172601,0.34474341725927676) circle (2.5pt);
\draw [fill=ududff] (3.867855231045659,2.3390353451257835) circle (2.5pt);
\draw [fill=ududff] (3.5685589684938592,-1.722842503791486) circle (2.5pt);
\end{scriptsize}
\end{tikzpicture}}
    \caption{The \v{C}ech complex recovers the homotopy type of the union of balls.
    Regarded as a Dowker complex, as described in Example \ref{example:DowkerCech}, the witnesses for a simplex are the intersections of the balls around the corresponding points.}
    \label{fig:Cech}
\end{figure}
We can view the \v{C}ech complex of some $X\subseteq (Z,d)$ as Dowker complex via
\[
    \Cech(X)_r = \Dow(X, Z, R_r),\qquad R_r = \{(x,z) \colon d(x,z)\leq r\}.
\]
Inspecting Figure \ref{fig:Cech}, we observe that the set of witnesses of a $k$-simplex is the intersection of the $r$-balls around the corresponding $k+1$ points.
We will impose the mass of these intersections as a second filtration parameter in the next section.

Similarly, for a finite metric space $(X,d)$, the intrinsic \v{C}ech complex is the Dowker complex $\mathcal{I}(X)_r = \Dow(X,X,\{d\leq r\})$, see for instance \cite{chazal_persistence_2014,brunDeterminingHomologyUnknown2023}.
\end{example} 


\section{The (Bi)filtrations}
\subsection{The Total Weight Filtration}\label{sec:TotalWeightFiltration}
As we have seen above, there are in general multiple witnesses for the presence of a simplex.
Let us count them:
\begin{definition}[{\cite[Definition 2]{robinson_cosheaf_2022}}]\label{def:TotalWeight}
The \emph{total weight} function is
\[
    t \colon \Dow(X,Y,R) \to \N\cup\{\infty\}, \; t(\sigma) = |\{y\in Y\colon \sigma\times\{y\}\subseteq R\}|.
\]
For $m\in\N$, we set $\Dow(X,Y,R)_m = \{\sigma \in \Dow(X,Y,R)\colon t(\sigma)\geq m\}$.
\end{definition}
Observe that we recover the whole Dowker complex for $m=1$, i.e. $\Dow(X,Y,R)_1 = \Dow(X,Y,R)$.
\begin{lemma}[{\cite[Proposition 2]{robinson_cosheaf_2022}}]\label{lemma:TotalWeightFiltration}
The superlevel sets of the total weight function $\Dow(X,Y,R)_m$ form a filtration by subcomplexes.
That is, for $m'\leq m \in ]0,\infty[$, we have $\Dow(X,Y,R)_{m} \subseteq \Dow(X,Y,R)_{m'}$.
\end{lemma}
\begin{proof}
First observe that if $\sigma \subseteq \tau$, the total weight satisfies $t(\sigma)\geq t(\tau)$, because any witness of $\tau$ in particular witnesses $\sigma$.
Hence, the superlevels of $t$ are indeed subcomplexes.
Moreover, they form a filtration because $t(\sigma)\geq m \Rightarrow t(\sigma) \geq m'$ as  $m'\leq m$.
\end{proof}
By this lemma, we view $\Dow(X,Y,R)_\bullet$ as a functor $]0,\infty[^{op}\to\Simp$.
While Dowker duality does not naively extend in general to a weak equivalence between total weight filtrations, $\Dow(X,Y,R)_\bullet \not\simeq\Dow(Y,X,R^\top)_\bullet$, we have the following result:

\begin{theorem}\label{thm:DowkerDualityTotalWeight}
Let $R\subseteq X\times Y$ be a relation.
Then we have a weak equivalence of filtrations $\vert\Dow(X,Y,R)_\bullet\vert \simeq  \vert\mathcal{S}(\Dow(Y,X,R^\top))_\bullet\vert$, where $\mathcal{S}$ is the subdivision filtration (Definition~\ref{def:SubdivisionBifiltration}).
Moreover, the weak equivalence is natural in the following sense:
If $T$ is a poset, $X,Y$ are fixed and $R_\bullet\colon T\to \Rel$ is a filtration of relations between $X$ and $Y$, then we have a weak equivalence of the two filtrations
\begin{align*}
    ]0,\infty[^{op}\times T &\to \Top,\\
    (m,t) &\mapsto  \vert\Dow(X,Y,R_t)_m\vert,\\
    (m,t) &\mapsto  \vert\mathcal{S}(\Dow(Y,X,R_t^\top))_m\vert.
\end{align*}
\end{theorem}

Before giving the proof, let us consider the following example which explains the general strategy.   
\begin{figure}
        \centering
        \begin{tikzpicture}

\node[] at (-1,0) {$%
            \bordermatrix{X\setminus Y & \alpha & \beta & \gamma & \delta  \cr%
                 \color{blue}a & 0 & 1 & 0 & 1 \cr%
                 \color{red}b & 0 & 0 & 1 & 1\cr%
                 \color{green}c & 0 & 1 & 1 & 0 \cr%
                 \color{cyan}d & 1 & 1 & 0 & 0}%
            $};%

        \node[] at (-0.5,-2) {$R$};
    
                   \draw[] (4.5,1) node[circle, inner sep = 1pt, radius=4pt, fill=blue!20] {$a$} 
                   -- (6,1) node[circle, inner sep = 1pt, radius=4pt, fill=blue!20] {$b$} 
                   -- (6,-1) node[circle, inner sep = 1pt, radius=4pt,  fill=blue!20] {$c$} 
                   -- (4.5, 1)   ;          
                           
                 \draw[fill=gray!50] (4.5,1) node[circle, inner sep = 1pt, radius=4pt, fill=blue!20] {$a$} 
                 -- (3,-1) node[circle, inner sep = 1pt, radius=4pt, fill=blue!20] {$d$} 
                 -- (6,-1) node[circle, inner sep = 1pt, radius=4pt,  fill=blue!20] {$c$} 
                 -- (4.5, 1)   ;          
                
        \node at (4.5,-2) {$\Dow(X,Y,R)$};

                \draw[] (8,-1) node[circle, inner sep = 1pt, radius=4pt, fill=blue!20] {$\alpha$} -- (9.5,-1) node[circle, inner sep = 1pt, radius=4pt, fill=blue!20](b) {$\beta$} ;
                \draw[] (9.5,-1) node[circle, inner sep=1 pt, radius=4pt, fill=blue!20]{$\beta$}  -- (11,-1) node[circle, inner sep = 1pt, radius=4pt, fill=blue!20]{$\gamma$} -- (9.5,1) node[circle, inner sep = 1pt, radius=4pt, fill=blue!20] {$\delta$} -- (9.5,-1); 
           
        \node at (9.5,-2) {$\Dow(Y,X,R^T)$};
\end{tikzpicture}
        \caption{The indicator matrix representing a binary relation $R$ (left; the colors correspond to the covering shown in Figure \ref{fig:Subdivision-of_Dual}); its Dowker complex $\Dow(X,Y,R)$, whose vertices are the row labels (middle); and the Dowker complex of the dual relation $\Dow(Y,X,R^\top)$.
        \label{fig:NewRelation}
        }
    \end{figure} 

\begin{example} 
Consider the relation $R$ in Figure~\ref{fig:NewRelation}, with the Dowker complex and the dual Dowker complex written alongside. 
The barycentric  subdivision  $\mathcal{S}(\Dow(Y,X,R^\top))$
is displayed in Figure~\ref{fig:Subdivision-of_Dual}.
Consider the following cover, in which the element indexed by $x\in X$ consists of those flags which are completely contained in the row labelled by $x$ in the indicator matrix:
\begin{align*}
A^a &= \{\{\beta\}, \{\delta\}, \{\beta,\delta\}, \{\{\beta\}\subsetneq \{\beta,\delta\} \} , \{\{\delta\}\subsetneq\{\beta,\delta\}\}   \} \\
A^b &= \{\{\gamma\}, \{\delta\}, \{\gamma,\delta\}, \{\{\gamma\}\subsetneq \{\gamma,\delta\} \} , \{\{\delta\}\subsetneq\{\gamma,\delta\}\}   \} \\
A^c &=  \{\{\beta\}, \{\gamma\}, \{\beta,\gamma\}, \{\{\beta\}\subsetneq \{\beta,\gamma\} \} , \{\{\gamma\}\subsetneq\{\beta,\gamma\}\}   \} \\
A^d &= \{\{\alpha\}, \{\beta\}, \{\alpha,\beta\}, \{\{\alpha\}\subsetneq \{\alpha,\beta\} \} , \{\{\beta\}\subsetneq\{\alpha,\beta\}\}   \}.
\end{align*}

\begin{figure}
        \centering
        \begin{tikzpicture}
    \node[circle, inner sep = 7pt, radius=4pt, fill=blue!20](a) at (4,-1) {$\alpha$};
    \node[circle, inner sep = 0pt, radius=4pt, fill=blue!20](ab) at (7,-1) {$\{\alpha,\beta\}$};
    \node[circle, inner sep = 7pt, radius=4pt, fill=blue!20](b) at (10,-1) {$\beta$};
    
    \draw (a) -- (ab) -- (b);
    
    \draw[] (10,-1) node[circle, inner sep = 7pt, radius=4pt, fill=blue!20](b){$\beta$}  
         -- (13,-1) node[circle, inner sep = 0pt, radius=4pt, fill=blue!20](bc){$\{\beta,\gamma\}$}  
         -- (16,-1) node[circle, inner sep = 7pt, radius=4pt, fill=blue!20](c){$\gamma$} 
         -- (13,1)  node(cd)[circle, inner sep = 0pt, radius=4pt, fill=blue!20] {$\{\gamma,\delta\}$} 
         -- (10,3)  node(d)[circle, inner sep = 7pt, radius=4pt, fill=blue!20] {$\delta$} 
         -- (10,1)  node(bd)[circle, inner sep = 0pt, radius=4pt, fill=blue!20] {$\{\beta,\delta\}$} 
         -- (10,-1); 
           
    \node[blue,draw=blue,inner sep=0pt,thick,ellipse,fit=(b) (bd) (d), rotate fit=0, text width = 1.1cm, label=left:$\color{blue}A^a$] {};
    \node[red,draw=red,inner sep=14pt,thick,ellipse,fit=(c) (cd) (d), rotate fit=56.31, text width = .5cm, label=right:$\color{red}A^b$] {};
    \node[green,draw=green,inner sep=0pt,thick,ellipse,fit=(b) (bc) (c), rotate fit=0, label=below:$\color{green}A^c$] {};
    \node[draw = cyan,ellipse,fit = (a)(b),thick, inner sep =0pt,label=below:$\color{cyan}A^d$]{};

\end{tikzpicture}
        \caption{Subdivision of  the Dowker complex of the dual relation, $\Dow(Y,X,R^\top)$, with the considered cover indicated; the colors correspond to the row labels in Figure~\ref{fig:NewRelation}.
        \label{fig:Subdivision-of_Dual}}
    \end{figure}

The nerve of the cover is isomorphic to  $\Dow(X,Y, R)$. Moreover, if we restrict the cover to flags in which the simplices are at least of dimension one, we obtain:

\begin{align*}
A^a_2 &= \{ \{\beta,\delta\} \} \\
A^b_2 &= \{ \{\gamma,\delta\}\}    \\
A^c_2 &=  \{  \{\beta,\gamma\}\}   \\
A^d_2 &= \{\{\alpha,\beta\}\}.
\end{align*}

The nerve of this cover consists of four isolated vertices, just like the weight two subcomplex $\Dow(X,Y,R)_2$.
Indeed, we will identify a nerve constructed in this fashion with the superlevel set of the total weight function.

\end{example}

\begin{proof}[Proof of Theorem~\ref{thm:DowkerDualityTotalWeight}] 
Let ${\cal A}=\{A^x\}_{x\in X}$ be the cover of $\mathcal{S}(\Dow(Y,X,R))$ given by the rows of the indicator matrix of $R$:
\[      A^x = \{\tau=(\tau_0\subsetneq\dots\subsetneq \tau_k) \in \mathcal{S}(\Dow(Y,X,R^\top)) \colon (x,y) \in R    \quad\textrm{for all}\quad  y\in \tau_i,\ i=0,\dots,k \}     \]
(note that we can abbreviate the condition to $\{x\}\times \tau_k\subseteq R$).
The cover member $A^x$ consists of all simplices $\tau$ that form the barycentric subdivision of the (possibly infinite-dimensional!) simplex in $\Dow(Y,X,R^\top)$ determined by $x$.
Hence (\cite[{Lemmas 3.2 and 3.3}]{LesnickNerveModels2024}) the geometric realization of each member of ${\cal A}$ is a geometric simplex and the intersections of the realizations are either empty or contractible. 
Hence the Functorial Nerve Theorem (Theorem~\ref{thm:SimplicialNerve}) applies, and we have a homotopy equivalence
\[     \phi_R\colon \mathcal{S}(\Dow(Y,X, R^\top)) \longrightarrow \Nerve{\cal A}          \]
The left hand side is filtered by the dimension of the minimal simplex in each flag, i.e. $\dim \tau_0$. Consider the filtration element $\mathcal{S}(\Dow(Y,X, R^\top))_m$ -- 
this consists of simplices $\tau=(\tau_0\subsetneq\dots\subsetneq \tau_k)$ such that $\dim \tau_0\ge m-1$, i.e. the simplex $\tau_0$ has at least $m$ elements. 
The subset of the $\Nerve{\cal A}$ corresponding to these $\tau$ consists of simplices $\sigma=\{x_0,\dots,x_n\}$ which have at least $m$ witnesses, hence those 
which belong to $\Dow(X,Y,R))_m$. 

Next, if $R_\bullet\colon T\to \Rel$ is a filtration of relations between $X$ and $Y$, then we get a diagram of covers $\mathcal{A}_{\bullet,\bullet}$ of the shape $]0,\infty[^{op}\times T$ where each cover is indexed by $X$.
Whenever $s\leq t\in T$ and $m'\leq m \in ]0,\infty[$, the cover element  $A^x_{s,m}$ is a subcomplex of the cover element  $A^x_{t,m'}$, for each $x\in X$.
Indeed, suppose we have a flag in $A^x_{m,s}$, that is $(\tau_0\subsetneq\dots\subsetneq \tau_k) \in \mathcal{S}(\Dow(Y,X,R^\top))$ satisfying $\dim(\tau_0)\geq m$ and $\{x\}\times \tau_k \in R_s$.
Then in particular $\dim(\tau_0)\geq m'$ and moreover, because $R_\bullet$ is a filtration, also $\{x\}\times \tau_k \in R_t$.
In other words, the structure maps of each $A^x$ (when viewed as a functor) are the inclusions of subcomplexes which are induced by restricting the structure map of the filtration $\mathcal{S}(\Dow(Y,X,R_\bullet^\top))_\bullet$.
The functorial part of the Nerve Theorem above (Theorem \ref{thm:SimplicialNerve}) gives us a commutative square
\[
\begin{CD}
      S(\Dow(Y,X, R_s^\top))_m           @>{\phi_R}>> \Nerve{{\cal A}_{m,s}}     \\        
          @VVV                                                          @VVV           \\
      S(\Dow(Y,X, R_t^\top))_{m'}          @>{\phi_R}>>  \Nerve{{\cal A}_{m',t}}     \\     
\end{CD}
\]
where the vertical maps are induced by the structure maps of the filtrations. 
\end{proof}

\begin{figure}
        \centering
            \begin{tikzpicture}
            \scriptsize

            \node[left,align=flush right] at (-2,0){$\Dow(X,Y,R)_m$};
            \node[left,align=flush right] at (-2,-3){$\mathcal{S}(\Dow(Y,X,R^T))_m$};

                \node at (0,1.5) {$m=1$};


                \draw[] (0.5,0) -- 
                        (0,{-sqrt(3)/2}) node[circle, fill=blue!20] {$d$} -- 
                        (-0.5,0);
                \draw[fill=gray!50] (0.5,0) node[circle, fill=blue!20] {$a$} -- 
                                    (0,{sqrt(3)/2}) node[circle, fill=blue!20](b) {$b$} --
                                    (-0.5,0) node[circle, fill=blue!20](c) {$c$} -- 
                                    (0.5,0);
              
                \node[blue, right] at(0.7,0){$\alpha\beta\eps$}; 
                \node[blue, left] at(-0.7,0){$\alpha\gamma\delta$};
                \node[blue, above right] at(0.2,{sqrt(3)/2}){$\alpha\delta\eps$};
                \node[blue, below right] at(0.2,{-sqrt(3)/2}){$\beta\gamma$};
            
                \node[blue, left] at(-0.4,0.5){$\alpha\delta$};
                \node[blue, right] at(0.4,0.5){$\alpha\eps$};
                \node[blue, left] at(-0.5,-0.4){$\gamma$};
                \node[blue, right] at(0.5,-0.4){$\beta$};

                \node[blue] at(0,0.3){$\alpha$};
                \node[blue,below] at(0,0){$\alpha$};

                \node at (0,-1.5) {\rotatebox{90}{$\simeq$}};

                
                \draw (-1,-3) to (0,-3.5);
                \draw (0,-3.5) to (1,-3);
                \draw[fill=gray!50] (0,-3)  -- 
                                    (-0.5,{sqrt(3)/2-3})  -- 
                                    (0.5,{sqrt(3)/2-3}) -- 
                                    (0,-3);
                \draw[fill=gray!50] (0,-3)  -- 
                                    (0.5,{sqrt(3)/2-3})  -- 
                                    (1,-3)  -- 
                                    (0,-3);
                \draw[fill=gray!50] (-1,-3)  -- 
                                    (-0.5,{sqrt(3)/2-3})  -- 
                                    (0,-3) -- 
                                    (-1,-3);
                \draw (-1,-3) -- 
                      (0.5,{sqrt(3)/2-3}) -- 
                      (0.5,-3);
                \draw (1,-3) -- 
                      (-0.5,{sqrt(3)/2-3}) -- 
                      (-0.5,-3);
                \draw (-0.75,{sqrt(3)/4-3}) -- 
                      (0,-3) -- 
                      (0.75, {sqrt(3)/4-3});
                \draw (0,{sqrt(3)/2-3}) -- (0,-3);

                \node at (4,1.5) {$m=2$};

        
                \node at (2,0) {$\supseteq$};

                \draw[] (4.5,0) node[circle, fill=blue!20] {$a$} -- 
                        (4,{sqrt(3)/2}) node[circle, fill=blue!20] {$b$} -- 
                        (3.5,0) node[circle, fill=blue!20] {$c$};
                \node[circle, fill=blue!20] at (4,{-sqrt(3)/2}) {$d$};

                \node at (4,-1.5) {\rotatebox{90}{$\simeq$}};

                \node at (2,-3) {$\supseteq$};

                \draw[fill] (4,-3.5) circle (1pt);
                \draw[fill=black] (4,{sqrt(3)/2-3}) circle[fill, radius = 1pt]  -- 
                                  (4,{sqrt(3)/3-3}) circle[fill, radius = 1pt];
                \draw[fill=black] (3.5,{sqrt(3)/6-3}) circle[fill, radius = 1pt]
                                -- (3.75,{sqrt(3)/4-3}) circle[fill, radius = 1pt] 
                                -- (4,{sqrt(3)/3-3}) circle[fill, radius = 1pt] 
                                -- (4.25,{sqrt(3)/4-3}) circle[fill, radius = 1pt] 
                                -- (4.5,{sqrt(3)/6-3}) circle[fill, radius = 1pt];
                \draw[fill=black] (3.25,{sqrt(3)/4-3}) circle[fill, radius = 1pt] -- 
                                  (3.5,{sqrt(3)/6-3}) circle[fill, radius = 1pt]-- 
                                  (3.5,-3)circle[fill, radius = 1pt];
                \draw[fill=black] (4.5,-3) circle[fill, radius = 1pt] -- 
                                  (4.5,{sqrt(3)/6-3})circle[fill, radius = 1pt] -- 
                                  (4.75, {sqrt(3)/4-3})circle[fill, radius = 1pt];

                \node at (8,1.5) {$m=3$};

                \node at (6,0) {$\supseteq$};

                \node[circle, fill=blue!20] at (7.5,0) {$c$};
                \node[circle, fill=blue!20] at (8,{sqrt(3)/2}) {$b$};
                \node[circle, fill=blue!20] at (8.5,0) {$a$};

                \node at (8,-1.5) {\rotatebox{90}{$\simeq$}};

                \node at (6,-3) {$\supseteq$};

                \draw[fill] (7.5, {sqrt(3)/6-3}) circle (1pt);
                \draw[fill] (8, {sqrt(3)/3-3}) circle (1pt);
                \draw[fill] (8.5, {sqrt(3)/6-3}) circle (1pt);
        
        \end{tikzpicture}
         \caption{Applying Theorem~\ref{thm:DowkerDualityTotalWeight} to the the complex from Example~\ref{example:DowkerDuality}.
         The top row is the filtration of $\Dow(X,Y,R)$ by total weight, the bottom row is Sheehy's subdivision filtration \cite{sheehy12multicover} applied to the dual Dowker complex $\Dow(Y,X,R^\top)$.}
         \label{fig:TotalWeightMulticoverExample}
    \end{figure}

See Figure~\ref{fig:TotalWeightMulticoverExample} for another illustration of the theorem.
One perspective on the previous theorem is that the subdivision filtration of a Dowker complex admits a smaller, thus computationally tractable, equivalent.

One can regard this total weight filtration of a Dowker complex as an instance of the multicover filtration introduced by Sheehy \cite{sheehy12multicover} (see Definition~\ref{def:MulticoverBifiltration} as well as \cite{cavannaWhenWhyTopological2017,blumberg_stability_2022}), where the cover is given by columns of the matrix that represents the relation.
By the multicover nerve theorem (Theorem \ref{thm:MulticoverNerveTheorem}; \cite{sheehy12multicover,cavannaWhenWhyTopological2017,blumberg_stability_2022}), the multicover filtration of $\Dow(X,Y,R)$ corresponds to the subdivision filtration of $\Dow(Y,X,R^\top)$ (this line of argument was presented in the first preprint version of this article and requires some technical finiteness assumptions about $R$).
We can conclude a version of the multicover nerve theorem in the setting of simplical complexes covered by simplices from our theorem:
\begin{corollary}[Multicover nerve theorem for covers by simplices]\label{cor:SimplicialMulticover}
Let $T$ be a poset, viewed as thin category.
Let $\mathcal{F}\colon T\to \Simp$ be a filtration of a simplicial complex and let $\mathcal{A} = \{A^i\}_{i\in I}$ be a family of functors $A^i\colon T\to \Simp$ such that
\begin{itemize}
    \item for all $t\in T$, the set $\{A^i_t\}_{i\in I}$ forms a cover of the complex $\mathcal{F}_t$ by simplices and
    \item for all $s\leq t \in T$ and all $i\in I$, the map $A^i_{s\leq t}$ is the restriction of $\mathcal{F}_{s\leq t}$ to $A^i_s$.
\end{itemize}
    Then we have a weak equivalence of filtrations 
    \begin{align*}
        \mathcal{M}(\mathcal{A}), \mathcal{S}(\Nerve{\mathcal{A}}) \colon ]0,\infty[^{op}\times T & \to \Simp,\\
        \mathcal{M}(\mathcal{A})_{m,t} & = \{\sigma \in \mathcal{F}_t \colon \sigma \textnormal{ is contained in at least }m\textnormal{ of the }A^i_t\},\\
    \end{align*}
\end{corollary}
Note that the condition that each $A^i_t$ is a simplex implies that intersections of cover elements are either empty or contractible, as they are again simplices.
\begin{proof}
    Let $V$ denote the set of vertices in the colimit $\colim \mathcal{F}$.
    Consider the filtration of relations $R_t\colon T\to \Rel$,
    \[
        R_t \subseteq V\times I,\qquad (v,i)\in R_t \Leftrightarrow v\in A^i_t.
    \]
    Then, by Theorem \ref{thm:DowkerDualityTotalWeight}, we obtain a weak equivalence of bifiltrations 
    \[
        \Dow(V,I,R_t)_m \simeq \mathcal{S}(\Dow(I,V,R_t^\top))_m.
    \]
    Now, $\Dow(V,I,R_t)_m = \mathcal{M}(\mathcal{A})_{m,t}$, as by definition a simplex is contained in either complex if and only if it is in at least $m$ of the $A^i_t$.
    Moreover, $\Dow(I,V,R_t^\top) = \Nerve{\mathcal{A}_t}$ because a non-empty intersection of some $A^i_t$ is equivalent to them having a common vertex $v\in V$.
\end{proof}
Similarly to how (functorial) Dowker duality is equivalent to (functorial) nerve theorems for covers consisting of simplices~\cite{bjornerTopologicalMethods1996,chowdhury_functorial_2018}, we extend this line of results to the multicover setting.
\begin{theorem}
    The statements of Theorem~\ref{thm:DowkerDualityTotalWeight} and Corollary~\ref{cor:SimplicialMulticover} are equivalent.
\end{theorem}
\begin{proof}
    It remains to show that the simplicial multicover nerve theorem implies our total weight Dowker duality.
    Indeed, using the notation of Theorem~\ref{thm:DowkerDualityTotalWeight}, given a filtration of relations we define a cover of the filtration of Dowker complexes $\mathcal{A} = \{A^y\}_{y\in Y}$ via
    \[
        A^y\colon T\to \Simp, \qquad A^y_t = \{\sigma \in \Dow(X,Y,R_t) \colon \sigma \times \{y\}\subseteq R_t\}.  
    \]
    By construction, $\mathcal{M}(\mathcal{A})_{\bullet,t} = \Dow(X,Y,R_t)_\bullet$.
    In addition, $\Nerve{\mathcal{A}} = \Dow(Y,X,R_\bullet^\top)$ which can be seen as follows:
    \begin{align*}
        &A^{y_0}_t\cap\ldots \cap A^{y_k}_t \neq \emptyset\\
        \Leftrightarrow& \exists x\in X \textnormal{ such that }x\in A^{y_0}_t\cap\ldots \cap A^{y_k}_t\\
        \Leftrightarrow& \exists x\in  X \textnormal{ such that }(x,y_{i})\in R_t \textnormal{ for all } i\in \{0,\ldots, k\}\\
        \Leftrightarrow& \{y_0,\ldots,y_k\} \in \Dow(Y,X,R^\top_t).\qedhere
    \end{align*}
\end{proof}

\subsection{The Measure Dowker Bifiltration}
For the remaining part, we focus on the case in which we have a one-parameter filtration of relations, so that adding the total weight into consideration gives rise to a bifiltration.
There can be uncountably many witnesses, as we saw in Example~\ref{example:DowkerCech} on the \v{C}ech complex.
To measure their size in such a case, we need the set $Y$ in the relation to be endowed with a measure.
If $Y$ is finite, we can take the counting measure to recover the total weight, as we will see in Example~\ref{example:MeasureDowkerTotalWeight}.

\begin{definition}\label{def:MeasureDowkerBifiltration}
Let $X$ be a set, $(Y,\Sigma,\mu)$ a measure space; let $\Lambda\colon X\times Y \to \R$ a function such that for all $x\in X$, the restricted map $y\mapsto \Lambda(x,y)$ is measurable (with respect to $\Borel(\R)$ on the codomain).
Define the \emph{measure Dowker bifiltration}
\begin{align*}
    \MDow(X,(Y,\Sigma,\mu),\Lambda) \colon ]0,\infty[^{op} \times [0,\infty[ \to & \Simp,
\end{align*}
where a non-empty finite subset $\sigma\subseteq X$ belongs to $\MDow(X,(Y,\Sigma,\mu),\Lambda)_{m,r}$ if and only if
\[
\mu(\{y\in Y \colon \Lambda(x,y)\leq 2r \textnormal{ for all }x\in\sigma\})\geq m.
\]
In words, we consider those simplices whose set of witnesses exceeds a certain mass where the filtered relation is given by sublevel sets of the function $\Lambda$.

In the special case in which $X$ and $Y$ are subsets of a common ambient metric space $(Z,d)$, we will use the shorthand notation
\[
    \MDow(X,\mu, \Lambda) = \MDow(X, (Y,\Borel(Y),\mu), \Lambda);
\]
if $\Lambda = d\left\vert\phantom{}_{X\times Y}\right.$, we omit it from the notation.
\end{definition}
In words, one includes $\sigma$ in $\MDow(X,\mu)_{m,r}$ if the intersection of the $2r$-balls centred at the points in $\sigma$ has at least measure $m$ with respect to $\mu$.
The factor $2$ appearing here will make it easier to phrase the stability results later on in Section \ref{sec:Stability}.

\begin{example}\label{example:MeasureDowkerCech}
    Fixing $r$, the complexes $\MDow(X,Z)_{\bullet,r}$ form a filtration of the \v{C}ech complex at scale $2r$.
    Recalling Example~\ref{example:DowkerCech} and Figure~\ref{fig:Cech}, the set of witnesses of a $k$-simplex is the intersection of the corresponding $k+1$ balls.
    With the new filtration parameter, we control precisely the measure of these intersections.
    Observe the relation to the measure bifiltration, in which we also include balls of radius $r$ if their mass exceeds a threshold; however, one does not impose further restrictions on the mass of the intersection there.
\end{example} 

\begin{lemma}\label{lemma:IsometryInvariance}
    Let $X$ be a subset of a metric measure space $(Z,d, \mu)$.
    If $(Z', d')$ is another metric space and $\varphi\colon (Z,d) \to (Z',d')$ is an isometry, then it induces an isomorphism of filtered simplicial complexes $\MDow(X,\mu)\xrightarrow[]{\cong} \MDow(\varphi(X), \varphi_{\#}(\mu))$.
\end{lemma}
\begin{proof}
    First fixing $m,r$, we want to show
    \begin{align*}
        [x_0,\ldots,x_k]\in \MDow(X,\mu)_{m,r}  
        & \Leftrightarrow [\varphi(x_0),\ldots,\varphi(x_k)] \in \MDow(\varphi(X), \varphi_{\#}(\mu))_{m,r}, \textnormal{ i.e. }\\
         \mu(\{ z\in Z \colon d(z,x_i) \leq 2r \textnormal{ for all } i \})\geq m 
        & \Leftrightarrow \varphi_{\#}\mu (\{ z'\in Z' \colon d'(z',\varphi(x_i)) \leq 2r \textnormal{ for all } i \})\geq m 
    \end{align*}
    To this end, we compute
    \begin{align*}
        &\varphi_{\#}\mu (\{ z'\in Z' \colon d'(z',\varphi(x_i)) \leq 2r \textnormal{ for all } i \})\\
        & =\mu(\varphi^{-1}(\{ z'\in Z' \colon d'(z',\varphi(x_i)) \leq 2r \textnormal{ for all } i \})\\
        & \overset{(\ast)}{=} \mu(\varphi^{-1}(\varphi(\{ z\in Z \colon d(z,x_i) \leq 2r \textnormal{ for all } i \})))\\
        &= \mu(\{ z\in Z \colon d(z,x_i) \leq 2r \textnormal{ for all } i \}).
    \end{align*}
    The last equality is due to $\varphi$ being bijective; the second to last equality $(\ast)$ requires some elaboration:
    In fact, $\varphi$ restricts to a bijection 
    \[
        \{ z\in Z \colon d(z,x_i) \leq 2r \textnormal{ for all } i \} \to \{ z'\in Z' \colon d'(z',\varphi(x_i)) \leq 2r \textnormal{ for all } i \}
    \]
    because it is an isometry.
    Indeed,
    \begin{align*}
        d(z,x_i) \leq 2r \textnormal{ for all } i  
        &\Rightarrow d'(\varphi(z),\varphi(x_i)) \leq 2r \textnormal{ for all } i, \\
        \textnormal{and } d'(z',\varphi(x_i)) \leq 2r \textnormal{ for all } i 
        & \Rightarrow d(\varphi^{-1}(z'),x_i) \leq 2r \textnormal{ for all } i.
    \end{align*}
    Finally, since the structure maps of the filtration are (by definition) inclusions of subcomplexes, they commute with the simplicial map induced by $\varphi$.    
\end{proof}

\begin{example}\label{example:MeasureDowkerTotalWeight}
    Let $X,Y$ be subsets of a common ambient metric space $(Z,d)$ and $\mu=\mu_Y$ be the counting measure of~$\,Y$.
    Then  $\MDow(X, \mu_Y)_{m,r} = \Dow(X,Y,R_{2r})_m$, where $R_r = \{(x,y)\colon d(x,y)\leq r\}$.
    In this way, the measure Dowker bifiltration generalises ordinary Dowker complexes endowed with the total weight filtration.
\end{example}

\begin{example}\label{example:GeometricMNeighborBifiltration}
Consider the empirical measure $\mu_X = \sum_{x\in X}\delta_x$ of a point cloud $X\subseteq \R^d$.
The measure Dowker bifiltration $\MDow(X,\mu_X)$, is a multineighbor complex \cite{babsonErdosRenyiGraphsLinialMeshulam2023} of a geometric graph with loops.
That is, given $r\geq0$ and $m\in\N$, consider the graph $G=(V,E)$ which has $V=X$ and edges $\{x,x'\}$ whenever $d(x,x')\leq 2r$ (note the relation to the 1-skeleton of the \v{C}ech complex!).
Here, we explicitly allow and even enforce the existence of a loop at each vertex.
The $m$-neighbor complex of this graph has $\sigma$ as a simplex if its vertices have $m$ common neighbors in $G$.
This is equivalent to saying $\mu_X(\{x' \in X \colon d(x,x')\leq 2r\textnormal{ for all }x\in\sigma\})\geq m$.

For an explicit example, consider $X$ to be the four vertices of the unit square in $\R^2$, the slice of its associated measure Dowker bilfiltration for $r=0.6$ is shown in Figure~\ref{fig:MeasureDowkerCech}.
Observe that non-trivial second homology appears for $m=1$, even though the point cloud is embedded in $\R^2$.
In contrast, Alpha or \v{C}ech complexes of points in the plane can have non-trivial homology only up to dimension $1$.

\begin{figure}
        \centering
        \subcaptionbox%
            {Consider $X$ to be four points on the vertices of a square in $\R^2$ and $\mu = \mu_X$ the associated counting measure.
            \label{fig:MeasureDowkerCech}}%
            {\resizebox{0.75\linewidth}{!}{\begin{tikzpicture}
            \draw[fill,fill opacity = 0.1] (1,1) circle (1.2);
            \draw[fill,fill opacity = 0.1] (1,2) circle (1.2);
            \draw[fill,fill opacity = 0.1] (0,1) circle (1.2);
            \draw[fill,fill opacity = 0.1] (0,2) circle (1.2);
            \node[circle, fill=blue!50] at (1,1) {};
            \node[circle, fill=blue!50] at (1,2) {};
            \node[circle, fill=blue!50] at (0,1) {};
            \node[circle, fill=blue!50] at (0,2) {};
            \node[align=center,below] at (0.5,-0.5) {$\bigcup_{x\in X}\overline{B}_{2r}(x)$\\$\simeq\Cech_{2r}(X) \cong \Delta^3$};

            \node at (4.5,3){$m=1$};

            \draw[fill, fill opacity=0.1] (4,1) -- (4,2) -- (5,1) -- (4,1);
            \draw[fill, fill opacity=0.1] (4,1) -- (4,2) -- (5,2) -- (4,1);
            \draw[fill, fill opacity=0.1] (4,1) -- (5,1) -- (5,2) -- (4,1);
            \draw[fill, fill opacity=0.1] (4,2) -- (5,2) -- (5,1) -- (4,2);
            \node[circle, fill=blue!50] at (4,1) {};
            \node[circle, fill=blue!50] at (4,2) {};
            \node[circle, fill=blue!50] at (5,1) {};
            \node[circle, fill=blue!50] at (5,2) {};
            \node[align=center,below] at (4.5,-0.5) {$\MDow(X,\mu_X)_{1,r}$\\$\cong sk^2(\Delta^3)$};

            \node at (7.5,3) {$m=2$};

            \draw (7,1) -- (7,2) -- (8,2) -- (8,1) -- (7,1) -- (8,2);
            \draw (7,2) -- (8,1);
            \node[circle, fill=blue!50] at (7,1) {};
            \node[circle, fill=blue!50] at (7,2) {};
            \node[circle, fill=blue!50] at (8,1) {};
            \node[circle, fill=blue!50] at (8,2) {};
            \node[align=center, below] at (7.5,-0.5) {$\MDow(X,\mu_X)_{2,r}$\\$\cong sk^1(\Delta^3)$};

            \node at (10.5,3) {$m=3$};

            \node[circle, fill=blue!50] at (10,1) {};
            \node[circle, fill=blue!50] at (10,2) {};
            \node[circle, fill=blue!50] at (11,1) {};
            \node[circle, fill=blue!50] at (11,2) {};
            \node[align=center, below] at (10.5,-0.5) {$\MDow(X,\mu_X)_{3,r}$\\$\cong \ast\sqcup\ast\sqcup\ast\sqcup\ast$};

        \end{tikzpicture}}}%
            \hfill%
        \subcaptionbox%
            {The subdivision intrinsic \v{C}ech complex (blue).
            \label{fig:SubdivisionTetrahedron}}%
            {\resizebox{0.22\linewidth}{!}{\tdplotsetmaincoords{180}{180}

\begin{tikzpicture}[scale=3]
\tdplotsetrotatedcoords{90}{0}{0}
\begin{scope}[tdplot_rotated_coords]
\coordinate (v0) at (0,0,0);
\coordinate (v1) at (1,0,0);
\coordinate (v2) at (0.5,0.866,0);
\coordinate (v3) at (0.5,0.2887,0.8165);

\coordinate (b01) at ($(v0)!0.5!(v1)$);
\coordinate (b02) at ($(v0)!0.5!(v2)$);
\coordinate (b03) at ($(v0)!0.5!(v3)$);
\coordinate (b12) at ($(v1)!0.5!(v2)$);
\coordinate (b13) at ($(v1)!0.5!(v3)$);
\coordinate (b23) at ($(v2)!0.5!(v3)$);

\coordinate (b012) at ($(v0)!0.333!(v1)!0.333!(v2)$);
\coordinate (b013) at ($(v0)!0.333!(v1)!0.333!(v3)$);
\coordinate (b023) at ($(v0)!0.333!(v2)!0.333!(v3)$);
\coordinate (b123) at ($(v1)!0.333!(v2)!0.333!(v3)$);

\coordinate (b0123) at ($(v0)!0.25!(v1)!0.25!(v2)!0.25!(v3)$);

\draw[gray!20] (v0) -- (v1) -- (v2) -- cycle;
\draw[gray!20] (v0) -- (v2) -- (v3) -- cycle;
\draw[gray!20] (v0) -- (v1) -- (v3) -- cycle;
\draw[gray!20] (v1) -- (v2) -- (v3) -- cycle;

\draw[thick, blue, opacity = 0.5]
    (b13) -- (b123) -- (b12) -- (b012) -- (b01) -- (b013) -- (b13);
\draw[thick, blue, opacity = 0.5]
    (b123) -- (b23) -- (b023) -- (b02) -- (b012);
\draw[thick, blue, opacity = 0.5]
    (b023) -- (b03) -- (b013);

\end{scope}
\node [] at (0.5,-0.5) {$\mathcal{S}(\Dow(X,X,R_r))_2$}; 
\end{tikzpicture}}}%
    \caption{The measure Dowker model for the subdivision intrinsic \v{C}ech complex.}
    \end{figure}
\end{example}
In the setting of this example, we can apply the total weight Dowker duality (Theorem~\ref{thm:DowkerDualityTotalWeight}) to obtain:
\begin{corollary}\label{cor:SubdivisionIntrinsicCech}
    Let $X$ be a finite metric space.
    We have a weak equivalence of filtrations $\MDow(X,\mu_X)_{m,r/2} \simeq \mathcal{S}(\mathcal{I}(X)_r)_m$.
\end{corollary}
\begin{definition}
The latter complex is the \emph{subdivision intrinsic \v{C}ech complex},
\[
    \mathcal{SI}(X)_{m,r} := \mathcal{S}(\Dow(X,X,\{d\leq r\}))_m.
\]
\end{definition}
We say that by Corollary~\ref{cor:SubdivisionIntrinsicCech}, the measure Dowker bifiltration is a model for the subdivision intrinsic \v{C}ech bifiltration.
This complex is not easy to draw in the example from Figure~\ref{fig:MeasureDowkerCech} on paper, we sketch it at level two in Figure~\ref{fig:SubdivisionTetrahedron}.

We conclude the section with the following relation between measure Dowker and degree Rips bifiltrations:
\begin{proposition}
Let $(Z,d)$ be a Polish space and let $\mu$ be a Borel measure on it.
We have the following interleaving:
\[
    \MDow(Z,\mu)_{m,r}\subseteq \DRips(Z,\mu)_{m,4r},\qquad \DRips(Z,\mu)_{m,r}\subseteq \MDow(Z,\mu)_{m,r}.
\]
\end{proposition} 
\begin{proof}
    Let $\sigma =[z_0,\ldots,z_k]\in \MDow(Z,\mu)_{m,r}$ be a $k$-simplex, i.e. $\mu\left( \bigcap_{i=0}^k\overline{B}_{2r}(z_i)\right)\geq m$.
    Then $\sigma\subseteq \mathcal{B}(\mu)_{m,4r}$ because for all $z_j\in\sigma$, we have 
    \[
        \mu\left(\overline{B}_{4r}(z_j)\right) \geq \mu\left( \bigcap_{i=0}^k\overline{B}_{2r}(z_i)\right)\geq m.
    \]
    As $m=0$ is excluded in Definition~\ref{def:MeasureDowkerBifiltration}, this intersection is in particular non-empty.
    This implies that the pairwise distances are bounded as $d(z_i,z_j)\leq 4r$ for all $z_i,z_j\in \sigma$ by the triangle inequality.
    Therefore, $\sigma\in \DRips(Z,\mu)_{m,4r}$.
    
    Vice versa, let $\sigma \in \DRips(Z,\mu)_{m,r}$ be a simplex.
    That is, $\mu(\overline{B}_r(z))\geq m$ for each $z\in\sigma$ and $d(z,z')\leq r$ for all $z,z'\in\sigma$.
    Then by the triangle inequality, $\overline{B}_r(z') \subseteq \bigcap_{z\in\sigma}\overline{B}_{2r}(z)$, as for any $z,z'\in\sigma$ and any $a\in \overline{B}_r(z')$, we have
    \[
        d(a,z)\leq d(a,z')+d(z',z)\leq 2r.
    \]
    Consequently, with any $z'\in \sigma$ we can estimate that
    \[
        \mu\left(\bigcap_{z\in\sigma}\overline{B}_{2r}(z)\right)\geq \mu(\overline{B}_r(z'))\geq m.\qedhere
    \]
\end{proof}
While one could also apply stability results of the degree Rips bifiltration now to obtain stability of the measure Dowker bifiltration, we will provide stronger bounds in a more general setting in the next section.

\section{Robustness and Stability}\label{sec:Stability}
This section is split into three parts.
First, we focus on the special case of Example~\ref{example:GeometricMNeighborBifiltration} and use the previous corollary to establish a robustness result.
In the second part, we consider general metric probability spaces and derive a density-sensitive stability theorem (Theorem~\ref{thm:DowkerStability}) which entails (a version of) the law of large numbers (Theorem~\ref{thm:LawOfLargeNumbers}) as a corollary.
Third, we apply the new stability theory to landmark-based bifiltrations, showing they are robust (Corollary~\ref{cor:RobustnessFixedLandmarks} and approximate the multicover bifiltration (Corollary~\ref{cor:LandmarksApproximateMulticover}).

\subsection{Empirical Probability Measure of a Finite Metric Space}\label{subsec:CountingMeasureIntrinsicCech}
In this section, we consider $(X,d)$ to be a finite metric space and endow it with its empirical probability measure $\nu_X$.
Recall the characterisation of the intrinsic \v{C}ech complex as a Dowker complex $\mathcal{I}(X)_r = \Dow(X,X, \{d\leq r\})$.
A normalised version of Corollary~\ref{cor:SubdivisionIntrinsicCech} reads as a weak equivalence of bifiltrations
\[
    |\mathcal{S}^n(\ICech(X)_r)_m|\simeq|\MDow(X,\nu_X)_{m,r/2}|.
\]
Thus, it remains to see that the subdivision intrinsic \v{C}ech complex approximates the subdivision Rips complex (which is robust) in the homotopy interleaving distance.
Indeed, this approximation result was obtained by Lesnick and McCabe \cite{LesnickNerveModels2024}, but it was not yet published at the time when this work first appeared as a preprint.
We rephrase it as follows:
\begin{proposition}[{\cite[Corollary 1.5 (i) and Remark 1.6 (iii)]{LesnickNerveModels2024}}]\label{prop:LesnickSIRobust}
    For $\delta>0$, consider the forward shift
\begin{align*}
    \alpha^\delta \colon \R^{op} \times [0,\infty[ & \to \R^{op} \times [0,\infty[\\
    (m,r) &\to (m-\delta,2r+\delta).
\end{align*}
Let $X_1, X_2$ be non-empty finite metric spaces.
Then $\mathcal{S}^n(\ICech(X_1))$ and $\mathcal{S}^n(\ICech(X_2))$ are $\alpha^\delta$-homotopy interleaved for all $\delta>d_{GHPr}((X_1,\nu_{X_1}), (X_2,\nu_{X_2}))$.
\end{proposition}
Now, our Corollary \ref{cor:SubdivisionIntrinsicCech} immediately yields:
\begin{corollary}\label{thm:MNeighborRobust}
    Consider two non-empty finite metric spaces endowed with their empirical probability measures $(X_1,d_1,\nu_{X_1}),(X_2,d_2,\nu_{X_2})$.
    Then the measure Dowker bifiltrations
    \[
        |\MDow(X_1,\nu_{X_1})|,\; |\MDow(X_2, \nu_{X_2})|
    \]
    are $\alpha^\delta$-homotopy interleaved functors for any $\delta >d_{GPr}((X_1,\nu_{X_1}),(X_2,\nu_{X_2}))$.
\end{corollary}

\begin{remark}
Our model for the subdivision intrinsic \v{C}ech complex has the advantage of being a bifiltration, whereas the model $\mathcal{NI}$ considered in \cite{LesnickNerveModels2024} is not: in the direction of the parameter $r$, the structure maps are not inclusions.
This is a ``semifiltration'' in their terminology.
\end{remark}

\subsection{General Metric Probability Spaces}
We present a stability result about the measure Dowker bifiltration, similar in spirit to the results of \cite{scoccolaLocallyPersistentCategories2020, rolle2023stable}.

\begin{theorem}\label{thm:DowkerStability}
    Suppose $(Z,d)$ is a  Polish space, endowed with Borel $\Sigma$-algebra $\Borel(Z)$.
    Let $X_1,X_2\in \Borel(Z)$ and let $\mu_1,\mu_2$ be measures on $(Z,\Borel(Z))$.
    Then we have
    \[
        d_{HI}(\MDow(X_1,\mu_1),\MDow(X_2,\mu_2))\leq \max(\{d_H(X_1,X_2), d_{Pr}(\mu_1,\mu_2)\}),
    \]
    where $d_H$ is the Hausdorff distance (Definition~\ref{def:HausdorffDistance}) and $d_{Pr}$ is the Prokhorov metric (Definition~\ref{def:Prokhorov}).
\end{theorem}
Our strategy is to take any $\delta \geq \max(\{d_H(X_1,X_2), d_{Pr}(\mu_1,\mu_2)\})$.
Then, because $\delta\geq d_H(X_1,X_2)$, the proximity relation $C=\{(x,y) \colon d(x,y)\leq \delta\}\subseteq X_1\times X_2$ has surjective canonical projections 
\(X_1\xleftarrow{\pi_{X_1}}C\xrightarrow{\pi_{X_2}}X_2\).
A relation with this feature is sometimes called a correspondence, hence the notation $C$ here.
We follow a proof strategy using Quillen's Theorem A which originates from Mémoli \cite{memoli2017distancefilteredspacestripods} and also appeared in \cite{blumbergUniversalityHomotopyInterleaving2023,scoccolaLocallyPersistentCategories2020}.
The key observation in our setting is the following lemma:

\begin{lemma}\label{lemma:CorrespondenceInducesInterleavingMorphism}
    Let $X_1\xleftarrow{\pi_{X_1}}C\xrightarrow{\pi_{X_2}}X_2$ be as  described above.
    For any subset $\sigma\subseteq X_1$, set $C(\sigma) = \pi_{X_2}(\pi_{X_1}^{-1}(\sigma))\subseteq X_2$.
    Then for any simplex $\sigma \in \MDow(X_1,\mu_1)_{m,r}$, every finite subset of $C(\sigma)$ is a simplex in $\MDow(X_2,\mu_2)_{m-\delta,r+\delta}$. 
\end{lemma}
\begin{proof}
Let $\tau\subseteq C(\sigma)$ be finite; it being a simplex in $\MDow(X_2,\mu_2)_{m-\delta,r+\delta}$ amounts to 
    \[
        \mu_2\left(\bigcap\limits_{y\in\tau}\overline{B}_{2(r+\delta)}(y)\right)\geq m-\delta.
    \]
    This holds true by the estimate
    \begin{align*}
        m-\delta &\leq \mu_1\left(\bigcap\limits_{x\in\sigma} \overline{B}_{2r}(x)\right)-\delta\\
            & \leq \mu_2\left(\left(\bigcap\limits_{x\in\sigma} \overline{B}_{2r}(x)\right)^\delta\right)\\
            & \leq \mu_2\left(\bigcap\limits_{y\in\tau} \overline{B}_{2(r+\delta)}(y)\right).
    \end{align*}
    Here, the first inequality is by definition of $\sigma \in \MDow(X_1,\mu_1)$; the second inequality is due to the definition of the Prokhorov metric and because $\delta\geq d_{Pr}(\mu_1,\mu_2)$; the third inequality is because
    \[
        \left(\bigcap\limits_{x\in\sigma} \overline{B}_{2r}(x)\right)^\delta\subseteq\bigcap\limits_{y\in\tau} \overline{B}_{2(r+\delta)}(y).
    \]
    Indeed, take $z'\in Z$ with $d(z,z')<\delta$ for some $z\in \bigcap\limits_{x\in\sigma}\overline B_{2r}(x)$.
    Let $y\in \tau$ be arbitrary and $x\in \sigma$ such that $(x,y)\in C$.
    This exists because $\delta\geq d_H(X_1,X_2)$ and $y\in C(\sigma)$.
    Finally, the triangle inequality yields
    \[
        d(y,z')\leq d(y,x)+d(x,z)+d(z,z') \leq \delta + 2r +\delta =2 (r+\delta).\qedhere
    \]
\end{proof}

Before we turn to the proof of Theorem \ref{thm:DowkerStability}, we recall a suitable version of Quillen's Theorem A:
\begin{lemma}[{\cite[{Lemma 6.3.3}]{scoccolaLocallyPersistentCategories2020}}]\label{lem:QuillenA}
    Let $K, L$ be abstract simplicial complexes and let $f\colon K\to L$ be a simplicial map which is surjective on vertices.
    Suppose $\sigma \in K \Leftrightarrow f(\sigma)\in L$.
    The we have a weak homotopy equivalence $|f|\colon |K|\xrightarrow{\simeq}|L|$.
\end{lemma}

\begin{proof}[Proof of Theorem~\ref{thm:DowkerStability}]

Let $\delta$ and $C$ be as before and let $[C]$ denote the simplicial complex of all non-empty subsets $\rho \subseteq C$.
We introduce two bifiltrations on $[C]$, namely, for $i\in\{1,2\}$ let
\begin{align*}
    \mathcal{F}^i \colon ]0,\infty[^{op}\times [0,\infty[ &\to \Simp,\\
    (m,r) &\mapsto \{\rho \in[C] \colon \pi_{X_i}(\rho)\in \MDow(X_i,\mu_i)_{m,r}\},
\end{align*}
where we introduced a slight abuse of notation
\[
    \pi_{X_i}(\rho) = \{\pi_{X_i}(z)\colon z\in\rho\}.
\]
Now, Lemma \ref{lem:QuillenA} guarantees that we have object-wise weak homotopy equivalences
\[
    \mathcal{F}^i_{m,r}\xrightarrow{\simeq}\MDow(X_i,\mu_i),\qquad \rho \mapsto \pi_{X_i}(\rho).
\]

Furthermore, $\mathcal{F}^1,\mathcal{F}^2$ are interleaved with respect to the forward shifts 
\[
    (m,r)\mapsto (m-\delta,r+\delta)
\]
by Lemma  \ref{lemma:CorrespondenceInducesInterleavingMorphism}:
If $\rho \in \mathcal{F}^1_{m,r}$, let $\sigma=\pi_{X_1}(\rho)\in \MDow(X_1,\mu_1)_{m,r}$.
Then $\pi_{X_2}(\rho)\subseteq C(\sigma)$, and therefore $\pi_{X_2}(\rho)\in \MDow(X_2,\mu_2)_{m-\delta,r+\delta}$ by the lemma.
But this is precisely the condition for $\rho\in \mathcal{F}^2_{m-\delta,r+\delta}$.
A symmetric argument with the indices $1$ and $2$ interchanged completes the proof.
\end{proof}

\begin{remark}\label{rmk:NoInterleavingOfSpaces}
    In the first preprint version of this article, a weaker result confined to an interleaving in homology was shown.
    That result is of course obtained from the homotopy interleaving in a standard way by applying the homology functor.
    In the meantime, a weak homotopy interleaving result of an equivalent bifiltration has appeared \cite[Proposition 7.6]{brun2024dualdegreecechbifiltration}.

    It is furthermore worth noting that the set $\{d(x,y)\leq\delta\}\subseteq X_1\times X_2$ takes on three different roles in our discussion:
    \begin{enumerate}
        \item It is the relation defining the Dowker complex.
        \item It is the correspondence inducing the interleaving maps in homology,
        \item It appears in the optimal transport characterization of the Prokhorov metric: this distance is the infimal $\delta$ such no more than $\delta$ of the mass needs to be transported over a distance greater than $\delta$, i.e. outside of $\{d(x,y)\leq\delta\}$.
        This is the optimal transport characterization via Strassen's theorem,
        \[
            d_{Pr}(\mu,\eta) = \inf\left\{\eps > 0 \colon \inf\limits_\gamma\{\gamma(\{d(x_1,x_2)\geq\eps\})<\eps
            \right\}
        \]
        where $\gamma$ ranges over all couplings $\gamma$ of $\mu$ and $\eta$ (i.e. $\pi_{1}^{\#}\gamma=\mu$, $\pi_{2}^{\#}\gamma=\eta$, where $\pi_1,\pi_2\colon X\times X \to X$ are the canonical projections).
    \end{enumerate}
\end{remark}

A direct consequence is the Gromov-Hausdorff-Prokhorov stability of the measure Dowker bifiltration of metric probability spaces:

\begin{corollary}\label{cor:GromovHausdorffProkhorovStability}
    For two metric probability spaces $(X_1,\Borel(X_1), \nu_1)$, $(X_2,\Borel(X_2), \nu_2)$, we have
    \[
        d_{HI}(\MDow(X_1,\nu_1),\MDow(X_2,\nu_2))\leq d_{GHPr}((X_1, \nu_1),(X_2, \nu_2)).
    \]
\end{corollary}

\begin{proof}
    Assume we have isometric embeddings into a common Polish space ${X_1\xrightarrow{\varphi} Z \xleftarrow{\psi} X_2}$.
Then
\begin{align*}
    &d_{HI}(\MDow(X_1,\nu_1),\MDow(X_2,\nu_2))\\
    & = d_{HI}(\MDow(\varphi(X_1),\varphi_{\#}(\nu_1)),\MDow(\psi(X_2),\psi_{\#}(\nu_2))\\
    &\leq \max\{d_H(\varphi(X_1),\psi(X_2)), d_{Pr}(\varphi_{\#}\nu_1, \psi_{\#}\nu_2)\}.
\end{align*} 
Here, we used Lemma~\ref{lemma:IsometryInvariance} for the first equality and Theorem~\ref{thm:DowkerStability} for the inequality.
As $\varphi, \psi$ are arbitrary, we can take the infimum over all such embeddings to get the desired assertion.
\end{proof}

As another consequence, we obtain a consistency result:
The homotopy interleaving distance between the measure Dowker bifiltration of a finite sample and the one of the true underlying metric probability space converges to zero in probability as the sample size goes to infinity.
This can be thought of as a `law of large numbers' -- the complex built on the empirical point sample converges to the true underlying bifiltration.
An analogous result for degree-Rips was previously known, see in particular \cite[Lemma 107]{rolle2023stable}, whose proof we follow closely.
Yet another result of similar type is Theorem 3.11 in \cite{blumberg_stability_2022}.

\begin{theorem}\label{thm:LawOfLargeNumbers}
    Let $(X,\mu)$ be a metric probability space with compact support $\textnormal{supp}(\mu)=A$.
    Let $(x_i)_{i\in\N}$ be an infinite sequence of i.i.d. samples from $\mu$.
    Set $X_n = \{x_1, \ldots, x_n\}$ and $\nu_{X_n} = n^{-1} \sum_{i=1}^n \delta_{x_i}$ the corresponding empirical probability measure.
    Then for all $\eps>0$,
    \[
        \lim_{n\to \infty}\P[d_{HI}(\MDow(X_n,\nu_{X_n}),\MDow(A,\mu))>\eps]=0.
    \]
\end{theorem}
\begin{proof}
     Let $\eps>0$.
    From the stability theorem (Theorem~\ref{thm:DowkerStability}) we have
    \begin{align*}
        \P[d_{HI}(\MDow(X_n,\nu_{X_n}),\MDow(A,\mu))>\eps]
        &\leq \P[\max(d_H(X_n, A), d_{Pr}(\nu_{X_n},\mu))>\eps].
    \end{align*}
    As $n\to\infty$, the empirical measures converge almost surely \cite[Theorem 11.4.1]{dudley_real_2002}, $\nu_{X_n}\to\mu$, and thus, also in the Prokhorov metric.
    Moreover, as $A$ is compact, there are $a_1,\ldots,a_N$ such that $A\subseteq B_{\eps/2}(a_1) \cup \ldots \cup B_{\eps/2}(a_N)$.
    Now each of these balls has positive mass under $\mu$, which is almost surely approximated by the empirical measures: 
    \[
    \nu_{X_n}(B_{\eps/2}(a_i)) \xrightarrow[n\to\infty]{a.s.}\mu(B_{\eps/2}(a_i)) > 0.
    \]
    Therefore, there are almost surely sample points falling into those balls, which means
    \[
        A\subseteq B_\eps(x_1)\cup\ldots\cup B_\eps(x_n)
    \]
    and consequently $\lim_{n\to\infty}\P[d_H(X_n, A) >\eps]=0$.   
    Thus 
    \[\P[\max(d_H(X_n, A), d_{Pr}(\nu_{X_n},\mu))>\eps]\to 0 \textnormal{ as } n\to \infty,\] as desired.
\end{proof}

\subsection{Application to Landmark-based Bifiltrations}
We now focus on Dowker complexes whose vertex set is a finite set of landmarks $L$ in some ambient Polish metric space $(Z,d)$.
In the one-parameter setting, such constructions have been successfully employed \cite{garlandExploringTopologyDynamical2016}; we now lay down the theoretical foundation for their use in two-parameter persistence.
As an immediate consequence of our stability theorem, we obtain a robustness result for Dowker complexes built on a fixed set of landmarks.
\begin{corollary}\label{cor:RobustnessFixedLandmarks}
    Consider a fixed finite set of landmarks in some metric space, $L\subseteq (Z,d)$.
    Given two finite point clouds $Y_1, Y_2\subseteq Z$, we have
    \[
        d_{HI}(\MDow(L,\nu_{Y_1}), \MDow(L,\nu_{Y_2})) \leq d_{Pr}(\nu_{Y_1}, \nu_{Y_2}),
    \]
    where $\nu_Y = \frac{1}{|Y|} \sum_{y\in Y} \delta_y$ is the empirical probability measure.
\end{corollary}

Moreover, our stability theorem also entails approximation guarantees for such landmark-based constructions.
Here, we need to assume that our ambient metric space is compact so that it is at a finite Hausdorff distance from the landmarks.
\begin{corollary}\label{cor:LandmarksApproximateMulticover}
    Consider a fixed finite set of landmarks in a compact metric space, $L\subseteq (Z,d)$.
    Then for any finite point cloud $Y\subseteq Z$, we have 
    \[
        d_{HI}(\MDow(L,\nu_{Y})_{\bullet,\bullet}, \M^n(Y)_{\bullet,2\bullet}) \leq d_H(L,Z),
    \]
    where $\M^n$ is the normalized multicover bifiltration (Definition \ref{def:NormalizedBifiltrations}).
\end{corollary}
\begin{proof}
Let $\eps > d_H(L,Z)$, then we have
    \begin{align*}
        &\MDow(L,\nu_Y)_{m,r}\\
        =&\MDow(L,\mu_Y)_{|Y|m,r}\\
        \overset{\eps}{\leftrightarrow}& \MDow(Z,\mu_Y)_{|Y|m,r}&\textnormal{by our stability theorem \ref{thm:DowkerStability}}\\
        =&\Dow(Z,Y,\{d\leq 2r\})_{|Y|m}\\
        \simeq&\mathcal{S}(\Dow(Y,Z,\{d\leq 2r\}))_{|Y|m}&\textnormal{by our duality theorem \ref{thm:DowkerDualityTotalWeight}}\\
        =&\SCech(Y)_{|Y|m,2r}\\
        \simeq&\M(Y)_{|Y|m,2r}&\textnormal{by the multicover nerve lemma \ref{thm:MulticoverNerveTheorem}}\\
        =&\M^n(Y)_{m,2r}.
    \end{align*}
    Here, we used $\overset{\eps}{\leftrightarrow}$ to denote an $\eps$-homotopy interleaving.
\end{proof}
Recall that in Euclidean ambient space, the mulitcover bifiltration can be equivalently computed via rhomboid tilings \cite{edelsbrunnerMultiCoverPersistenceEuclidean2021, corbetComputingMulticoverBifiltration2023} which is of polynomial size ($\Theta(|Y|^{d+1}$ simplices for $Y\subsetneq \R^d$), but usually still too large to be practical.
Hence, our landmark-based construction provides a new practical avenue to tackle bipersistence computations.

\section{Computational Results}\label{sec:ComputationalResults}
For the algorithmic aspects, we focus on the discrete-combinatorial version of the measure Dowker bifiltration, i.e. we assume $X$ and $Y$ are finite and endow $Y$ with its counting measure $\mu_Y$.
Note that the measure Dowker bifiltration is \emph{multi-critical}, which means that a simplex need not appear at a unique minimal bidegree in $]0,\infty[^{op}\times [0,\infty[$ but rather at a collection of mutually incomparable bidegrees.
In order to compute them, we make use of the following characterization:
\begin{lemma}\label{lemma:BidegreeOfAppearance}
    Let $X,Y$ be finite sets, $\mu_Y$ the counting measure of $Y$ and $\Lambda \colon X\times Y \to \R$ a function.
    The simplex $\sigma =[x_0,\ldots,x_k] \in \MDow(X,\mu_Y,\Lambda)$ appears in bidegrees
    $\{(m, r_m(\sigma))\}_{1\leq m\leq m_{max}}$ where $r_m(\sigma)$ is the {$m$\th} smallest value of $\left\{ \max\limits_{i} \frac{1}{2}\Lambda(x_i,y) \colon y\in Y\right\}$ and $m_{max}$ is the maximal value of $m$ for which $r_m(\sigma)$ exists.
\end{lemma}
Note that this collection of bidegrees is in general not minimal as one might have $r_m(\sigma) = r_{m'}(\sigma)$ for $m\neq m'$.
For instance, if $\Lambda \equiv 0$, our lemma yields the bidegrees $\{(m,0)\}_{1\leq m \leq |Y|}$ for any simplex.
\begin{proof}
    The simplex appears as soon as there are $m$ witnesses.
    In other words, 
    \[
        r_m(\sigma) = \min\{r>0 \colon \mu_Y(\{y\colon \Lambda(x,y) \leq 2r \textnormal{ for all }x\in\sigma\})\geq m \}.
    \]
    Now the map $r\mapsto \mu_Y(\{y\colon \Lambda(x,y) \leq 2r \textnormal{ for all }x\in\sigma\})$ is monotonically increasing.
    Moreover, there is some $r_0\in \R$ such that this function evaluates to $0$ for all $r<r_0$.
    Hence, it reaches the value $m$ after increasing $m$ times (not necessarily at $m$ distinct $r$-values).
\end{proof}
Observe that the only way for a simplex to be critical, i.e. to have a single bidegree of appearance is if $y\mapsto\max\{\frac{1}{2}\Lambda(x,y) \colon x\in\sigma\}$ is a constant function.

We use the preceding lemma to construct a list of simplices with their appearances recursively, adapting the classical algorithm of \cite{zomorodian_fast_2010} to Algorithm~\ref{alg:TWBifilteredDowker}.
Knowing the witnesses of a simplex $\sigma$, we go to a coface $\tau = \sigma \cup \{j\}$ of codimension $1$.
Its witnesses are given by the intersection of the witnesses of $\sigma$ and those of $\{j\}$.
The bidegree of appearance is computed by sorting the first entries, up to some specified $m_{max}$, of $\left\{ \max\limits_{i} \frac{1}{2}\Lambda(x_i,y) \colon y\in Y\right\}$.

\SetKwComment{Comment}{/* }{ */}

\begin{algorithm}
\caption{Computing the bifiltered measure Dowker complex.}\label{alg:TWBifilteredDowker}
\KwIn{A finite set $X$ of size $n$ with elements labelled $0$ through $n-1$, a finite set $Y$, a matrix $\Lambda \in \R^{X\times Y}$, $m_{max}\in \N$, $\dim_{max}\in \N$.}
\KwOut{A list of simplices of $\Dow(X,Y,\Lambda)$ with bidegrees of appearance.}
$\texttt{SimplexList} \gets []$ \Comment{global variable}
\For{$k=n-1$ \KwTo $0$}{
  $\texttt{AppendUpperCofaces}(\{k\}, \Lambda[k]))$\Comment{$\Lambda[k]$ denotes k\th row}
}
\textbf{return} $\texttt{SimplexList}$\;
\SetKwFunction{FMain}{AppendUpperCofaces}
    \SetKwProg{Fn}{Function}{:}{}
    \Fn{\FMain{$\sigma$, $\texttt{WitnessValues}$}}{
        $\texttt{sorted} \gets \texttt{SmallestElements}(\texttt{WitnessValues}, m_{max})$\;
        $\texttt{Appearances} \gets \{(\texttt{sorted}[i]/2,i) \colon 0<i<m_{max}, \texttt{sorted}[i]\leq r_{max}\}$\;
        $\texttt{SimplexList} \gets \texttt{SimplexList} \cup (\sigma, \texttt{Appearances})$\;

        \If{$\dim(\sigma)\leq \dim_{max}$}{
            \For{$j = \max(\sigma)+1$ \KwTo $n-1$}{
                $\tau \gets \sigma \cup \{j\}$\;
                $\texttt{CommonWitnessValues} \gets (\max\{\texttt{WitnessValues}[i],\Lambda[j][i]\})_{i\in \{0,\ldots,|Y|\}}$\;
                $\texttt{AppendUpperCofaces}(\tau, \texttt{CommonWitnessValues})$\;
            }
        }
    }
\end{algorithm}

Let us briefly discuss runtime and size aspects.
In the worst case there is $y\in Y$ such that $X\times \{y\}\subseteq R_r$ for some $r$, which means that the Dowker complex will be a filtration of the complete simplex on $X$, which has $2^{|X|}-1$ simplices.
Consequently, its $\dim_{max}$-skeleton has $O(|X|^{\dim_{max}+1})$ simplices.
For each simplex, we have to store up to $m_{max}$ bidegrees of appearance.
They are computed by sorting the first $m_{max}$ entries of an array of size $|Y|$, which is known as the partial sorting problem and can be implemented via a combination of heap-select and heap-sort giving complexity $O(|Y|\log(m_{max}))$.
This leads to a total run-time of $O(|X|^{\dim_{max}+1}\cdot |Y|\cdot \log(m_{max}))$ for the skeleton of the bifiltered Dowker complex.
For small values of $\dim_{max}$, as one needs for low-dimensional persistent homology, we found this to be computionally tractable.
The computational bottleneck in our experiments is consistently the homology computation, although we admit that the RIVET software \cite{rivet} we employed for its ease of use is not state of the art in terms of speed, which is \cite{bauerEfficientTwoParameterPersistence2023a}.
In the case of Euclidean proximity as the relation, it might be interesting to speed up the construction using a geometric data structure for storing nearest neighbors.
Even more interesting would be to decrease the size of the complex in a way similar to how Alpha complexes are much smaller but equivalent to \v{C}ech.
Note that the naive approach of just intersecting with the Alpha complex at scale $2r$ does indeed change the homotopy type, as can be observed in the example of Figure~\ref{fig:MeasureDowkerCech}:
When $m=1$, $\MDow$ has non-trivial second homology and this stays true if we wiggle the points minimally to move into general position.
But the Alpha complex of points in $\R^2$ cannot have any second homology.

Before we conclude the chapter, we present some computational results, which we hope do not just illustrate the ideas presented in this work, but also will stimulate further applications of the measure Dowker bifiltration.

\begin{example}\label{example:NoisyAnnulus}
    Inspired by the experiment of \cite[Appendix A]{blumberg_stability_2022}, we consider three point clouds in the plane, illustrated in Figure~\ref{fig:TwoAnnuli}.
\begin{itemize}
    \item $X$ contains 256 points uniformly sampled from an annulus with inner radius 0.4 and outer radius 0.5,
    \item $Y$ contains 256 points, where $95\%$ are sampled from the same annulus, and $5\%$ of the points sampled uniformly from the disk of radius 0.4.
    \item $Z$ consists of 256 points sampled uniformly from the disk of radius 0.5.
\end{itemize}
    We consider the measure Dowker complex of an equispaced $10\times 10$-grid $S$ with respect to the counting measures from Corollary~\ref{cor:RobustnessFixedLandmarks} (Figure~\ref{fig:H_1(MDow(Grid))}) as well the $m$-neighbor bifiltration from example~\ref{example:GeometricMNeighborBifiltration} (Figure~\ref{fig:H_1(MDow(X,X)}) with $m$ up to $50$.
    Then we compute the Hilbert function, that is the dimensions of $H_1$ 
    \[
        \hf[H_1(\MDow(X,\mu_X))]\colon ]0,\infty[^{op}\times [0,\infty[ \to \N; \; (m,r) \mapsto \dim(H_1(\MDow(X,\mu_X)_{m,r},\Z/2)).
    \]
    on a $50\times 50$ grid $G\subset ]0,\infty[^{op}\times [0,\infty[$, using the RIVET software \cite{rivet}.
    
    For comparison, we repeat the same computations using the degree Rips bifiltration with results displayed in Figure~\ref{fig:H_1(DR)}, as in the original experiments of \cite{blumberg_stability_2022} and further studied in \cite{rolleDegreeRipsComplexesAnnulus2022}.
    In Table \ref{tab:AnnulusComputeTime}, we summarize the combined time (in seconds) needed to both construct the bifiltration (up to dimension 2) and compute 1-dimensional persistent homology.
    The landmark-based construction is the fastest
    \begin{table}[]
        \centering
        \begin{tabular}{lrrr}
\toprule
 & $X$ & $Y$ & $Z$ \\
\midrule
Landmarks & 63.077231 & 64.993127 & 62.795267 \\
Dowker model for $\mathcal{I}$ & 3356.893599 & 3499.155293 & 3484.669348 \\
degree-Rips & 126.224289 & 120.733327 & 196.133863 \\
\bottomrule
\end{tabular}
        \caption{Time elapsed (in seconds) for the computation of the bidegrees of appearance up to dimension two and $H_1$ for different point clouds and bifiltrations.}
        \label{tab:AnnulusComputeTime}
    \end{table}
    
\begin{figure}
    \centering
    \begin{subfigure}[b]{\textwidth}
        \includegraphics[width=0.9\linewidth]{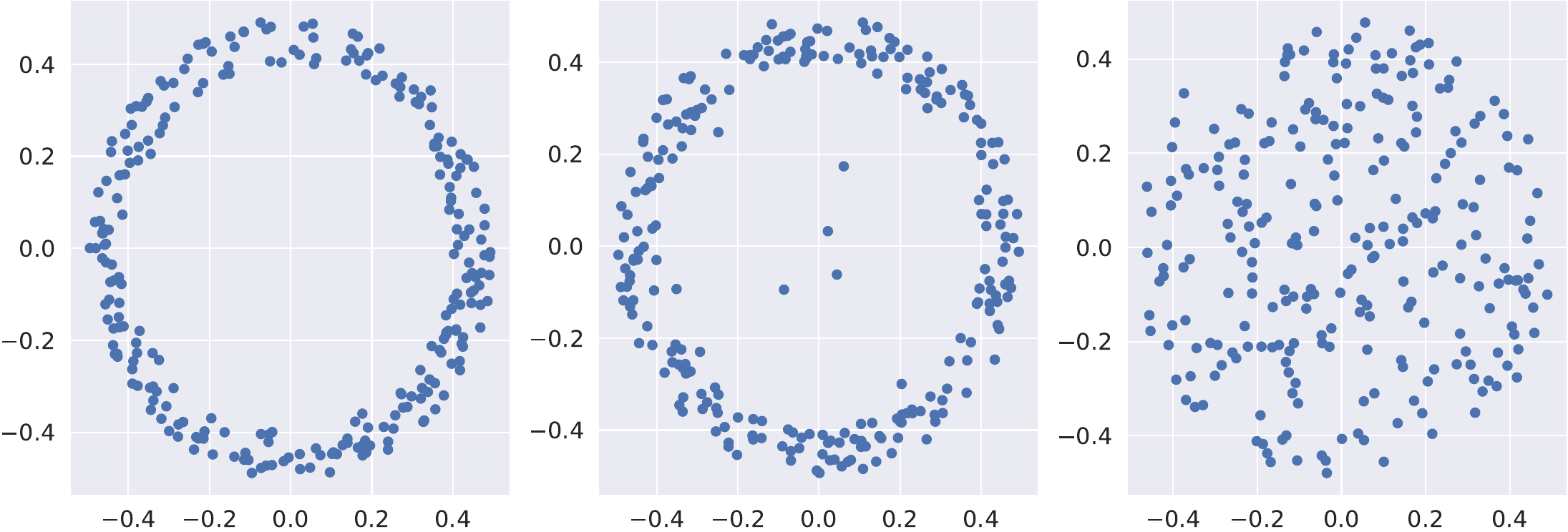}
    \caption{Two noisy annuli $X$ (left) and $Y$ (middle) as well as a uniform sample $Z$ (right) from a disk.}
    \label{fig:TwoAnnuli}
    \end{subfigure}
    \begin{subfigure}[b]{\textwidth}
       \includegraphics[width=0.9\linewidth]{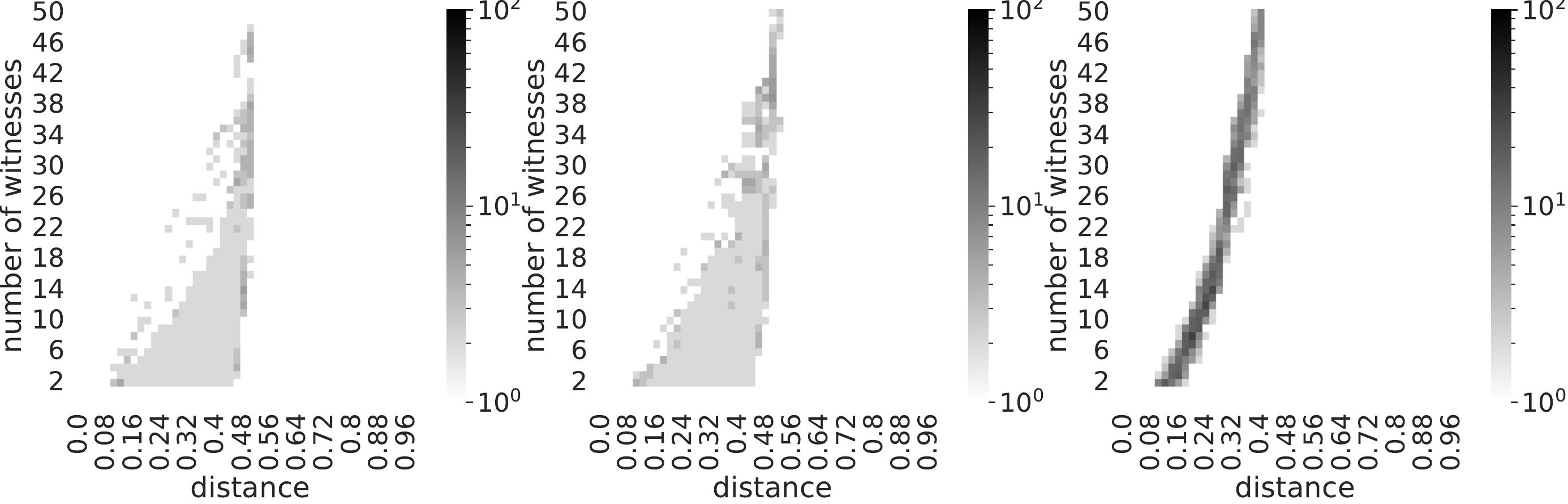}
    \caption{The Hilbert functions of $H_1(\MDow(S,\mu_X))$ (left), $H_1(\MDow(S,\mu_Y))$ (middle), and $H_1(\MDow(S,\mu_Z))$ (right) sampled on a grid, where $S\subseteq [-\frac{1}{2},\frac{1}{2}]^2$ is an equispaced $10\times10$ grid of landmarks.
    \label{fig:H_1(MDow(Grid))}}
    \end{subfigure}
    \begin{subfigure}[b]{\textwidth}
        \includegraphics[width=0.9\linewidth]{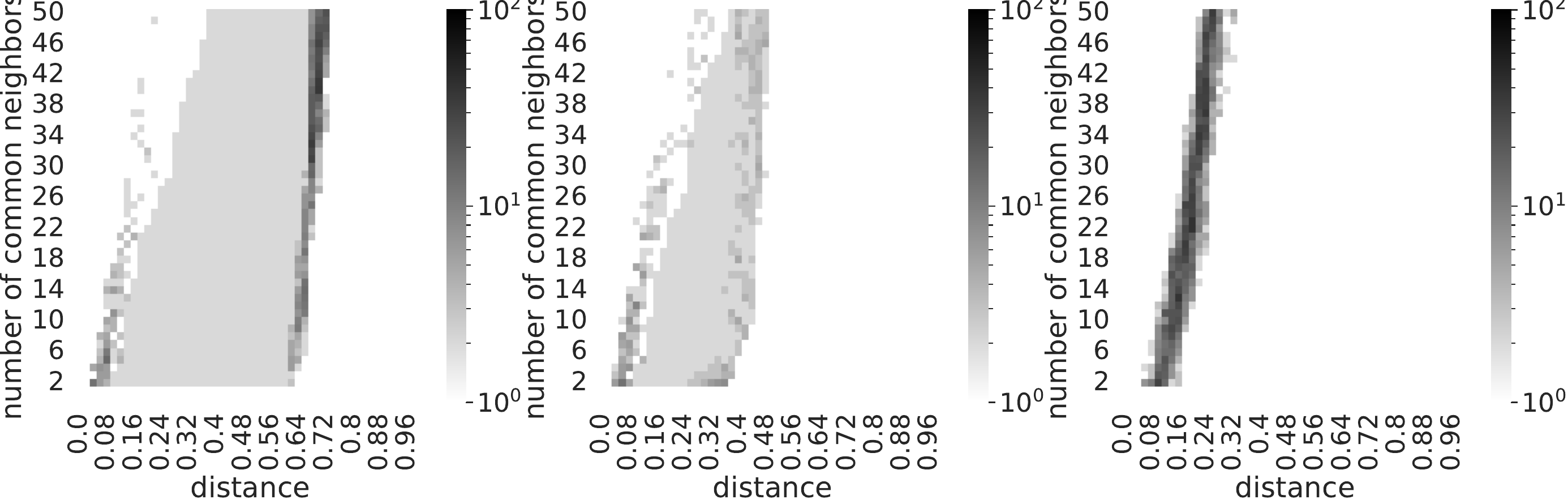}
    \caption{The Hilbert functions of our Dowker model of subdivision intrinsic \v{C}ech, $H_1(\MDow(X,\mu_X))$ (left), $H_1(\MDow(Y,\mu_Y))$ (middle), and $H_1(\MDow(Z,\mu_Z))$ (right).}
    \label{fig:H_1(MDow(X,X)}
    \end{subfigure}
    \begin{subfigure}[b]{\textwidth}
        \includegraphics[width=0.9\linewidth]{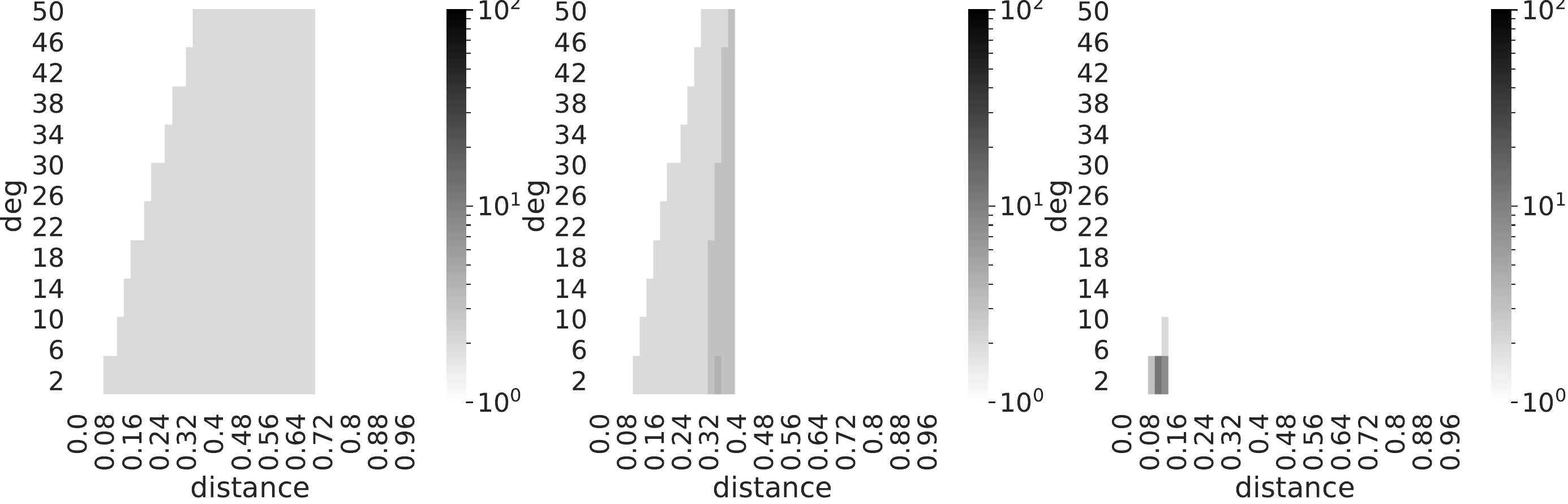}
    \caption{The Hilbert functions of the degree Rips bifiltrations, $H_1(\DRips(X, \mu_X))$ (left), $H_1(\DRips(Y,\mu_Y))$ (middle), and $H_1(\DRips(Z,\mu_Z))$ (right).}
    \label{fig:H_1(DR)}
    \end{subfigure}
    \caption{The results of the computations from Example~\ref{example:NoisyAnnulus}.
    Each Hilbert function is evaluated on an equispaced $50\times 50$ grid.
    Note that the color-scale is logarithmic.}
\end{figure}
\end{example}

\begin{example}\label{example:ERHypergraph}
    Consider $\Lambda$ to be a matrix with i.i.d. uniform entries from $[0,1]$.
    If we fix a sublevel set of $\Lambda \leq p$ as the relation, we obtain an Erdös-Renyi hypergraph in the sense of \cite{barthelemyClassModelsRandom2022}.
    We can keep track of the dimension of homology as $p$ varies,
    cf. Figure~\ref{fig:ErdosRenyi}.
    Studying vanishing thresholds for this two-parameter persistent homology is an intriguing direction for future research.
    A first step in this direction can be seen in \cite{babsonErdosRenyiGraphsLinialMeshulam2023} in the setting of $m$-neighbor complexes of Erdös-Renyi graphs.
    \begin{figure}
        \centering
        \includegraphics[width=1\textwidth]{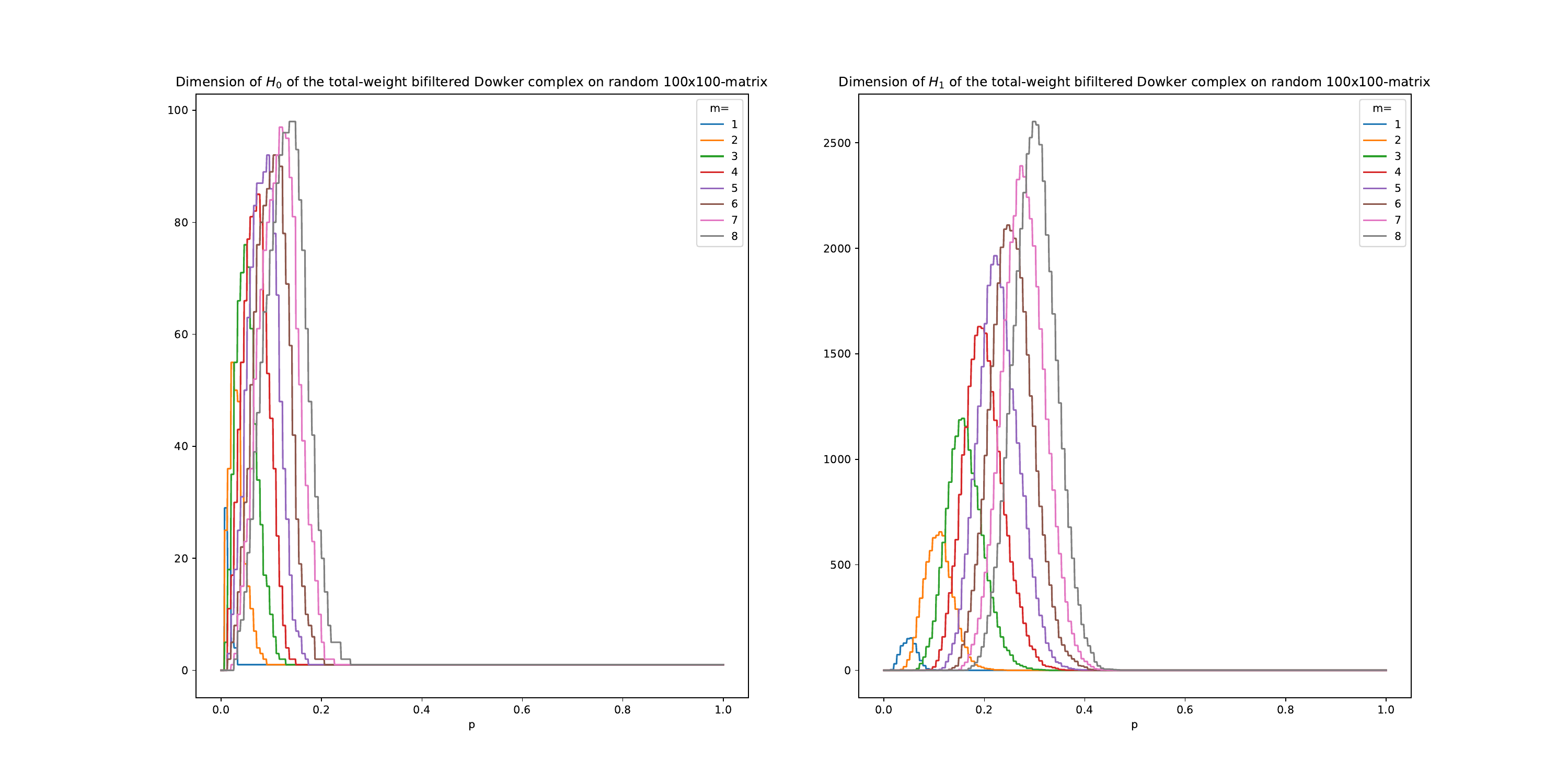}
        \includegraphics[width=1\textwidth]{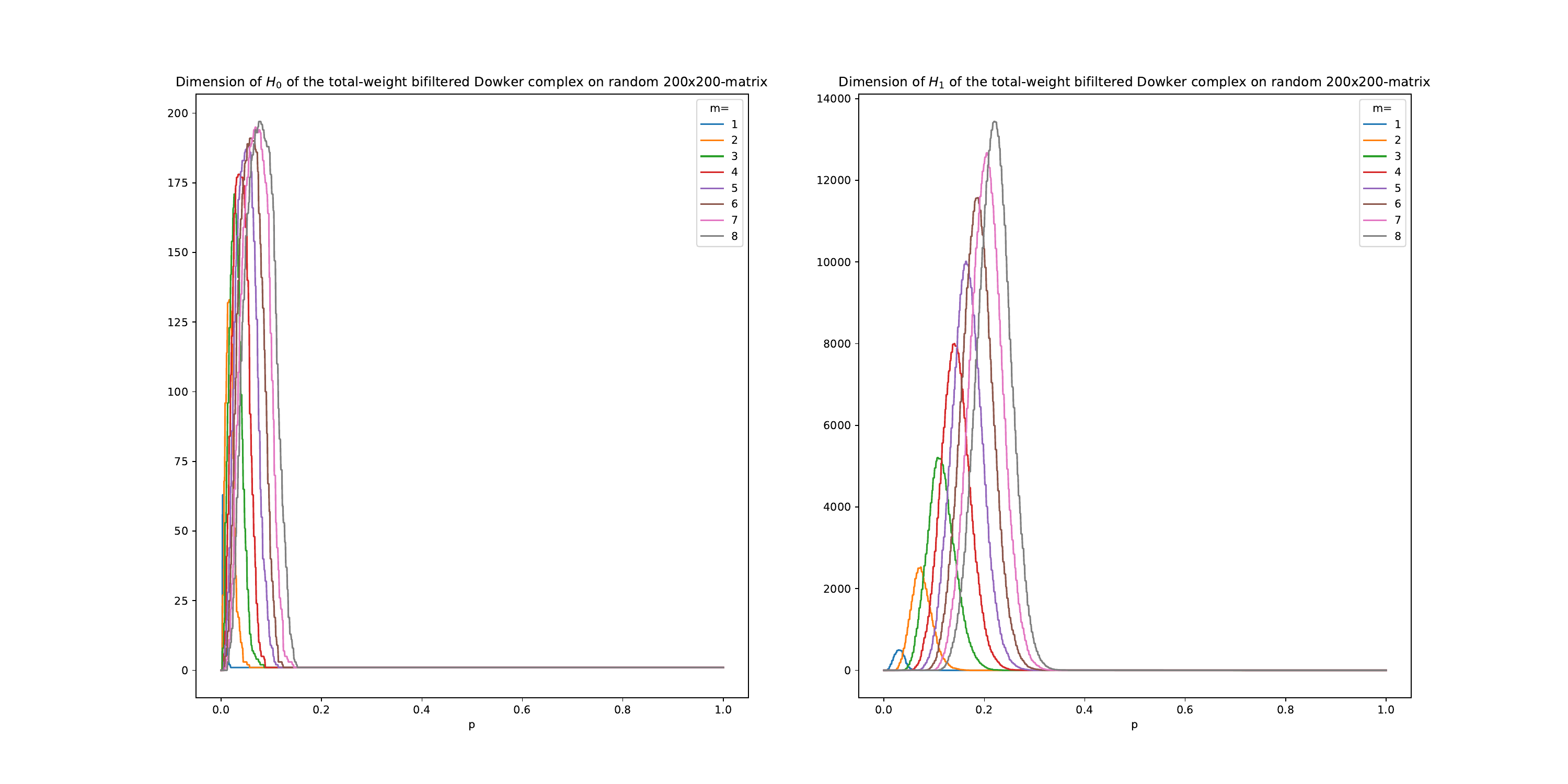}
        
        \caption{Hilbert functions of $H_0$ (left) and $H_1$ (right) for $n\times n$ matrices with random uniform entries ($n=100$ (top), $n=200$ (bottom)) for some small values of $m$.}
        \label{fig:ErdosRenyi}
    \end{figure}
\end{example}


\begin{example}\label{example:GeneExpression}
Consider the dataset \cite{misc_gene_expression_cancer_rna-seq_401} of gene expressions from 20531 genes of 801 patients with five different types of cancer.
We regard this as 801 points in $\R^{20531}$; however, the Euclidean distance is not very meaningful due to the curse of dimensionality \cite{Bellman_Curse_1957}.
Instead, we consider the $k$-nearest neighbor matrix with respect to the cosine distance.
Explicitly, the cosine distance between $x_1,x_2\in \R^d$ is 
\[
d_C(x_1,x_2) = 1 - \frac{\langle x_1,x_2\rangle}{\|x_1\|\|x_2\|}.
\]
The filtration of relations  
\[
    R_k = \{(x_1,x_2)\colon |\{x\in X \colon d_C(x_1,x)<d_C(x_1,x_2)\}|\leq k\} 
\]
is then encoded by the sublevel sets of the matrix
\[
    \Lambda_{ij} = k \Leftrightarrow j \textnormal{ is the }k\textnormal{\th nearest neighbor of }i.
\]
Note that both filtration parameters are on the same scale, as opposed to degree Rips for innstance, which is conceptually nice and might help to interpret the results.
We compute $H_0$ of the bifiltered Dowker complex for the number of nearest neighbors $k$ up to $64$ and the total weight up to $64$.
In other words, two patients end up in the same connected component, which we interpret as a cluster, if and only if they have $m$ common points among their respective $2k$ nearest neighbors.
Of course, for $m>2k$, there are no points at all.
We can inspect the appearance and merging of clusters using the Hilbert function, shown in the left panel of Figure~\ref{fig:GeneCancerClustering}.
Moreover, we can visualize the data based on a force-directed graph layout of the 1-skeleton of the Dowker complex for a given choice of $k$ and $m$.
In the right panel of Figure~\ref{fig:GeneCancerClustering}, this is done for $k=30, m=12$ at the top, which yields the true number of clusters 5. 
The colors represent the true label, i.e. which type of cancer the patient has.
When we set $k=60,m=20$ in the bottom right panel, we only get three connected components, yet the cluster structure remains visible.
This hints at a connection between Dowker complexes and dimensionality reduction techniques like UMAP, which builds a graph on a high dimensional point cloud by looking at the distance to the $k$-nearest neighbor and embeds it using a force-directed layout \cite{mcinnesUMAPUniformManifold2020}.
\begin{figure}
    \centering
        \includegraphics[width=\textwidth]{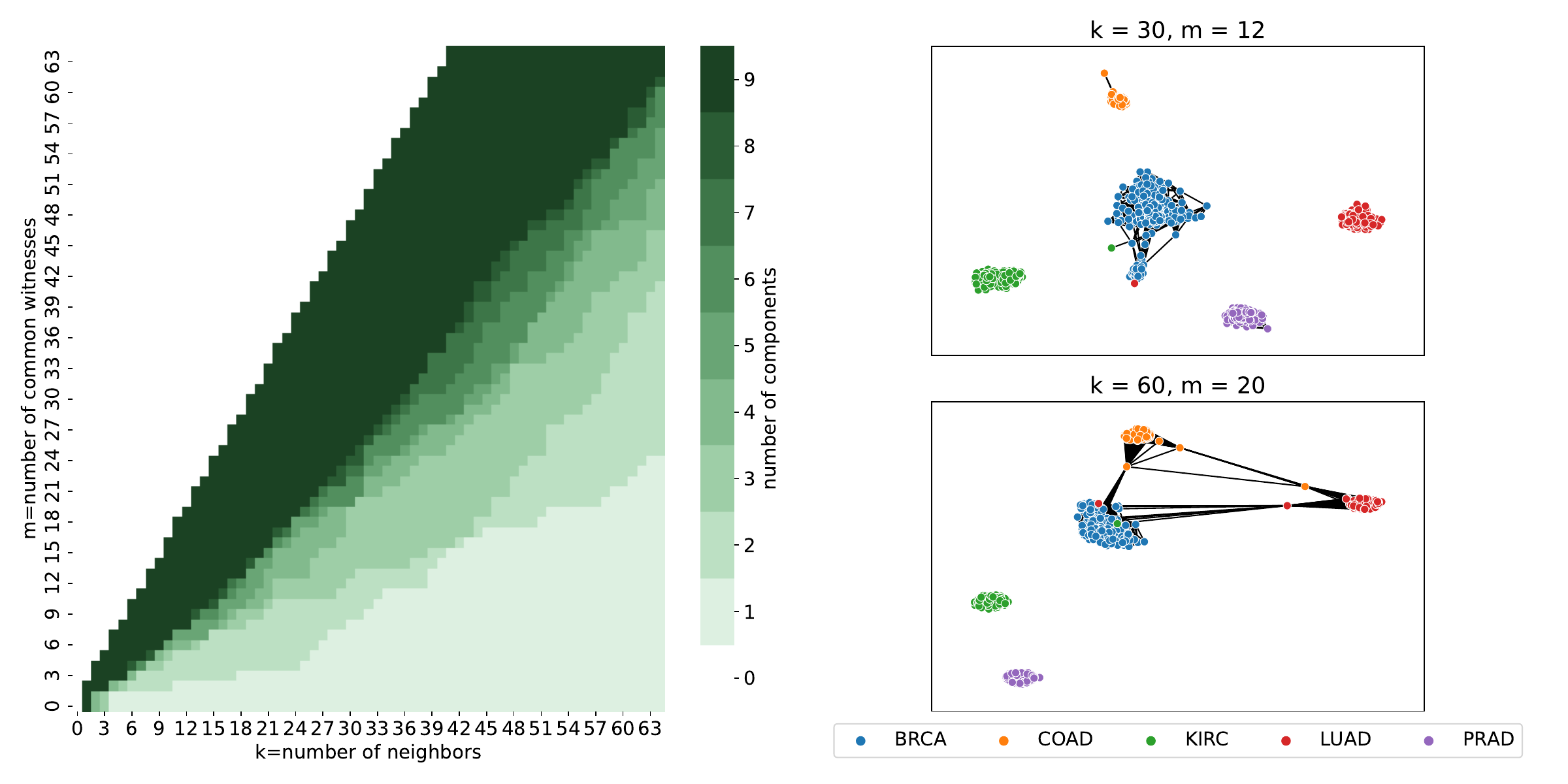}
    \caption{The gene expression clustering and dimensionality reduction of Example~\ref{example:GeneExpression}.
    We show $\hf$, where $M = H_0((\MDow(X,X,\Lambda)))$, in the left panel, and two particular choices of the 1-skeleton of $\MDow(X,X,\Lambda)_{m,k}$ embedded using a force-directed layout on the right panel.
    In the left panel, the color gradient corresponds to the number of connected components, where the darkest shade of green includes everything $\geq 9$.
    In the right panel, the colors encode the true label, i.e. the type of cancer as shown in the legend.}
    \label{fig:GeneCancerClustering}
\end{figure}
In a similar vein, Rolle \& Scoccola provide an extensive study of clustering via two-parameter zeroth persistent homology and their accompanying \texttt{persistable} software\cite{rolle2023stable}.
Using our Dowker model for the subdivision intrinsic \v{C}ech complex in their pipeline we produce the example of Figure~\ref{fig:persistable-clustering}
\begin{figure}
    \centering
    \includegraphics[width=0.5\linewidth]{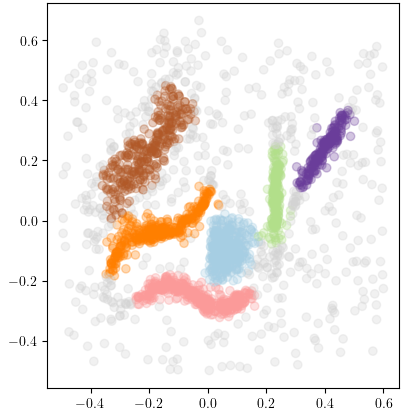}
    \caption{A clustering produced by the measure Dowker bifiltration ($R_r = \{(x,y)\colon \|x-y\|\leq r$) in conjunction with \texttt{persistable} using the slice with y-offset $32$ and angle $\arctan(-100)$ with the `conservative' flattening.
    Dataset from \cite{McInnes_2017}.}
    \label{fig:persistable-clustering}
\end{figure}
\end{example}

\begin{example}\label{exaple:ProteinLigandBindingAffinityPrediction}
    Dowker and neighborhood complexes have previously been used with great success by Liu et al. for predicting protein-ligand binding affinity \cite{liu_neighborhood_2021, liu_dowker_2022}, a task in computer aided drug design.
    We follow the setup of Liu et al. to create the complex, which is common in both referenced works, just that we have an additional filtration parameter.
Given a protein-ligand pair, build a bipartite Dowker complex which has the ligand atoms of a fixed kind as vertices and use protein atoms of fixed type as witnesses.
We use all possible combinations of ligand atoms from $\{C,N,O,S,P,F,Cl,Br,I\}$ and protein atoms from $\{C, N, O, S\}$.
Then we proceed to compute persistent homology in dimensions 0 and 1 of each such Dowker complex, bifiltered by both distance (up to 100 ångströms) and total weight (up to $16$).
The persistence modules are then vectorized via the Hilbert functions; we concatenate all the vectors to obtain a ``topological fingerprint'' of the protein-ligand pair.
Note that \cite{liu_neighborhood_2021, liu_dowker_2022} use more sophisticated vectorization methods based on persistent Laplacians.
However, we are mainly interested in the question how much additional information the introduction of the second (i.e. total weight) filtration carries.
We train a random forest regression with the ``PDBbind-refined'' dataset using the library ``DeepChem'' accompanying the book \cite{ramsundar_deep_2019}.
The test data is ``PDBbind-core'', for which we report the prediction accuracy in table~\ref{tab:AffinityPrediction}.
Notably, we observe slightly higher accuracy for two-parameter than for the one-parameter setup (which corresponds to the column m=1).
However, we are not able to reproduce the even much higher scores of \cite{liu_dowker_2022, liu_neighborhood_2021}, who use the persistent Laplacians and more sophisticated vectorizations, capturing more information that just the Hilbert function.
\begin{table}
    \centering
    \begin{footnotesize}
    \begin{tabular}{c||c|c|c|c|c|c|c|c|}
        $m_{max}$ &     1   &     2   &   3   &   4   &   5   &   6   &   7   &   8   \\\hline
        train $R^2$ & 0.96 & 0.96 & 0.95 & 0.96 & 0.96 & 0.96 & 0.96 & 0.96 \\
        test $R^2$ & 0.45 & 0.44 & 0.47 & 0.51 & 0.52 & 0.49 & 0.51 & 0.51 \\
    \multicolumn{9}{c}{} \\
         $m_{max}$ &       9   &   10  &   11  &   12  &   13  &   14  &   15  &   16 \\\hline
        train $R^2$ & 0.96 & 0.95 & 0.95 & 0.95 & 0.96 & 0.95 & 0.95 & 0.95 \\
        test $R^2$ & 0.53 & 0.50 & 0.53 & 0.50 & 0.49 & 0.48 & 0.47 & 0.52 \\
    \end{tabular}
    
    \end{footnotesize}
    \caption{Pearson $R^2$ of the binding affinity prediction of Example~\ref{exaple:ProteinLigandBindingAffinityPrediction}}
    \label{tab:AffinityPrediction}
    
\end{table}
\end{example}


\section*{Acknowledgements}
The authors thank Paweł Dłotko for introducing them to each other and encouraging them to collaborate, as well as for helpful comments on a preliminary version of this work.
NH thanks Michael Lesnick, for giving a course on multiparameter persistence, and for helpful conversations, some of which happened during the Dagstuhl seminar 24092 \textit{Applied and Combinatorial Topology}.
There, Lesnick presented results pertaining to models of the subdivision-Rips filtration via intrinsic \v{C}ech (cf. the discussion at the start of Section \ref{subsec:CountingMeasureIntrinsicCech}); we thank him for sharing a preliminary version of those results with us \cite{LesnickNerveModels2024} and for helpful discussions about the relation to our work.
Moreover, NH thanks Lars Salbu for pointing him to Dowker complexes at the YRN Meeting on Topology and Applications, organized by Julian Brüggemann and funded by the Bonn International Graduate School of Mathematics.
Julian Brüggemann also provided valuable comments on a preliminary version of this article.
Part of this work, the application to protein-ligand binding affinity prediction, was carried out while NH was visiting Helmholtz Munich.
He is grateful for the hospitality of Bastian Rieck and his group, as well as for the financial support from the University of Warsaw under the IDUB initiative, POB 3 action IV.4.1.
He is also indebted to Alexander Rolle for discussions about the degree-Rips bifiltration during the visit to Munich and for guidance with the persistable software.
NH was supported thorough Paweł Dłotko's grant by the Dioscuri program initiated by the Max Planck Society,  jointly managed with the National Science Centre (Poland), and mutually funded by the Polish Ministry of Science and Higher Education and the German Federal Ministry of Education and Research. 
\printbibliography
\end{document}